\documentclass{amsart}
\usepackage[utf8]{inputenc}
\usepackage{commath}
\usepackage{xurl}
\usepackage[maxnames=5,maxalphanames=5,style=alphabetic-verb]{biblatex}
\bibliography{exampleref.bib}
\usepackage[margin=1in]{geometry}
\usepackage{graphicx}
\usepackage{amsmath,amsfonts,amssymb,amsthm,etoolbox}
\usepackage{latexsym,hyperref}
\usepackage[most]{tcolorbox}
\usepackage{xcolor}
\usepackage{scrextend}
\usepackage{enumitem}
\usepackage[normalem]{ulem}
\usepackage{mathtools}
\usepackage[linewidth=1pt]{mdframed}

\AtBeginBibliography{\sloppy\setlength{\emergencystretch}{2em}}

\title{On Zarankiewicz's bounds for valued vector spaces}

\author{Hongyi Gou}
\address{Department of Mathematics, National University of Singapore, Singapore} 
\email{e0708226@u.nus.edu}

\author{Mihir Mittal}
\address{Department of Mathematics, National University of Singapore, Singapore}
\email{mihirmittal24@u.nus.edu}

\author{Chieu-Minh Tran}
\address{Department of Mathematics, National University of Singapore, Singapore} 
\email{trancm@nus.edu.sg}

\author{Zhenyu Yang}
\address{School of Statistics and Data Science, Nankai University, Tianjin, China} 
\email{2110314@mail.nankai.edu.cn}
\keywords{Zarankiewicz problem, valued vector spaces, semilinear hypergraphs, semiequational theories, incidence geometry}
\subjclass[2020]{Primary 05C35, 03C98; Secondary 03C45, 03C60}

\newcommand\restr[2]{{
  \left.\kern-\nulldelimiterspace 
  #1 
  \littletaller 
  \right|_{#2} 
  }}

\theoremstyle{plain}
\newtheorem{defn}{Definition}[section]

\newtheorem{lem}[defn]{Lemma}

\newtheorem{manualthm}{Theorem}

\newtheorem{prop}[defn]{Proposition}
\newtheorem*{prop*}{Proposition}
\newtheorem*{thm*}{Theorem}
\newtheorem{thm}[defn]{Theorem}
\newtheorem{cor}[defn]{Corollary}

\newtheorem*{claim*}{Claim}

\theoremstyle{remark}
\newtheorem{rem}[defn]{Remark}
\theoremstyle{remark}

\theoremstyle{remark}

\theoremstyle{remark}
\newtheorem{exmp}[defn]{Example}
\theoremstyle{remark}

\theoremstyle{remark}

\theoremstyle{remark}

\theoremstyle{remark}

\theoremstyle{remark}

\numberwithin{equation}{section}

\makeatletter
\let\save@mathaccent\mathaccent
\newcommand*\if@single[3]{%
  \setbox0\hbox{${\mathaccent"0362{#1}}^H$}%
  \setbox2\hbox{${\mathaccent"0362{\kern0pt#1}}^H$}%
  \ifdim\ht0=\ht2 #3\else #2\fi
  }
\newcommand*\rel@kern[1]{\kern#1\dimexpr\macc@kerna}
\newcommand*\widebar[1]{\@ifnextchar^{{\wide@bar{#1}{0}}}{\wide@bar{#1}{1}}}
\newcommand*\wide@bar[2]{\if@single{#1}{\wide@bar@{#1}{#2}{1}}{\wide@bar@{#1}{#2}{2}}}
\newcommand*\wide@bar@[3]{%
  \begingroup
  \def\mathaccent##1##2{%
    \let\mathaccent\save@mathaccent
    \if#32 \let\macc@nucleus\first@char \fi
    \setbox\z@\hbox{$\macc@style{\macc@nucleus}_{}$}%
    \setbox\tw@\hbox{$\macc@style{\macc@nucleus}{}_{}$}%
    \dimen@\wd\tw@
    \advance\dimen@-\wd\z@
    \divide\dimen@ 3
    \@tempdima\wd\tw@
    \advance\@tempdima-\scriptspace
    \divide\@tempdima 10
    \advance\dimen@-\@tempdima
    \ifdim\dimen@>\z@ \dimen@0pt\fi
    \rel@kern{0.6}\kern-\dimen@
    \if#31
      \overline{\rel@kern{-0.6}\kern\dimen@\macc@nucleus\rel@kern{0.4}\kern\dimen@}%
      \advance\dimen@0.4\dimexpr\macc@kerna
      \let\final@kern#2%
      \ifdim\dimen@<\z@ \let\final@kern1\fi
      \if\final@kern1 \kern-\dimen@\fi
    \else
      \overline{\rel@kern{-0.6}\kern\dimen@#1}%
    \fi
  }%
  \macc@depth\@ne
  \let\math@bgroup\@empty \let\math@egroup\macc@set@skewchar
  \mathsurround\z@ \frozen@everymath{\mathgroup\macc@group\relax}%
  \macc@set@skewchar\relax
  \let\mathaccentV\macc@nested@a
  \if#31
    \macc@nested@a\relax111{#1}%
  \else
    \def\gobble@till@marker##1\endmarker{}%
    \futurelet\first@char\gobble@till@marker#1\endmarker
    \ifcat\noexpand\first@char A\else
      \def\first@char{}%
    \fi
    \macc@nested@a\relax111{\first@char}%
  \fi
  \endgroup
}
\makeatother

\begin{document}
\begin{abstract}
We establish absolute and relative almost-linear Zarankiewicz bounds for semilinear relations in valued vector spaces. For every fixed arity and description complexity, a $K_{t,\ldots,t}$-free semilinear $r$-partite hypergraph has at most
\[
 O\!\left(n^{r-1}(\log n)^c\right)
\]
edges, where $c$ depends only on the arity and the number of valuative literals. In the bipartite case a separate arbitrary-trace argument gives the explicit bound $O(n(\log n)^{2s})$ for description complexity $(\rho,s)$. We also prove a relative extension theorem: intersecting any relation with a hereditary almost-linear profile by $s$ affine moving-radius comparisons increases the logarithmic exponent by at most $2s$. For the additive affine-valuative structures on $\mathbb Q_p$ and $\mathbb C_p$, quantifier elimination converts these semilinear results into bounds for all definable relations. Finally, over every valued field with infinite value group, we construct $K_{2,2}$-free semilinear point--box graphs of description complexity $(1,4)$ with $\Omega(n\log n/\log\log n)$ edges.
\end{abstract}
\maketitle

\section{Introduction}

Given integers $n_1, n_2 \ge 1$ and $t \ge 2$, \emph{Zarankiewicz's problem}~\cite{zarankiewicz1951problem} asks for the maximum number of edges in a $K_{t,t}$-free bipartite graph $G = (V_1, V_2; E)$ with $|V_1| = n_1$ and $|V_2| = n_2$. Writing $n := n_1 + n_2$, the classical theorem of K\H{o}v\'ari, S\'os, and Tur\'an~\cite{kovari1954problem} gives
\[
|E| = O_t\!\left(n^{2-\frac{1}{t}}\right),
\]
and this is essentially tight in general~\cite{erdos1964extremal}. In more structured settings, however, significantly stronger bounds are possible. The prototypical example comes from incidence geometry: taking $V_1$ to be a set of $n_1$ points and $V_2$ a set of $n_2$ lines in $\mathbb{R}^2$, with edges given by incidence, the graph is $K_{2,2}$-free, so the K\H{o}v\'ari--S\'os--Tur\'an bound yields $|E| = O(n^{3/2})$. The celebrated Szemer\'edi--Trotter theorem~\cite{szemeredi1983extremal} improves this to the sharp bound $|E| = O(n^{4/3})$, revealing a substantial \emph{power saving} over the general case.

This improvement can be understood through the lens of model theory: the incidence relation between points and lines in $\mathbb{R}^2$ is expressible as a bilinear equation in the coordinates, so the incidence graph $G$ is induced by a relation definable in the real field $(\mathbb{R};+,\times)$. Fox, Pach, Sheffer, Suk, and Zahl~\cite{fox2017semi} showed that definability is the essential reason for the power saving, proving that any $K_{t,t}$-free bipartite graph induced by a definable (equivalently, semi-algebraic) relation in $(\mathbb{R};+,\times)$ satisfies a Szemer\'edi--Trotter-type bound. These ideas were further developed by Basu, Chernikov, and others~\cite{basu2018minimal,chernikov2020cutting}, who identified model-theoretic \emph{distality} as the general setting for such power-saving results.

A natural question is whether other model-theoretic assumptions can yield even sharper bounds. Evidence in this direction comes from Basit, Chernikov, Starchenko, Tao, and the third author~\cite{basit2021zarankiewicz}, who proved almost-linear bounds for semilinear $r$-partite hypergraphs: any $K_{t,\cdots,t}$-free $r$-partite hypergraph induced by a semilinear relation of description complexity $(\rho,s)$ over $\mathbb{R}$ satisfies
\[
|E| = O_{r,\rho,s,t}\!\bigl(n^{r-1} (\log n)^{s (2^{r-1} - 1)}\bigr).
\]
Since $(\mathbb{R};+,<)$ is the prototypical example of a \emph{linear} structure in model theory, this result suggests that almost-linear Zarankiewicz bounds should hold under a suitable purely model-theoretic linearity assumption. A proposed abstract approach to achieving this assumption uses semi-equations. Let $\mathcal F$ be a family of subsets of a set $X$. It is \emph{laminar} if any two members are disjoint or comparable by inclusion. For $k\geq2$, it is \emph{$k$-wise laminar} if, whenever $k$ members have nonempty common intersection, two of them are comparable. A partitioned formula $\phi(x;y)$ is a \emph{$(k,1)$-semiequation} if the family of fibres
\[
 \{\phi(M^{|x|};b):b\in M^{|y|}\}
\]
is $k$-wise laminar. A theory is $(k,1)$-semiequational if every formula is a finite Boolean combination of such formulas. Chernikov and Mennen~\cite{ChernikovMennen} observed that $(2,1)$-semiequations are closely related to the basic relations used in semilinear Zarankiewicz arguments, and asked how far this connection extends for larger $k$.

However, \((k,1)\)-semiequationality does not encompass all natural examples of model-theoretic linearity. Consider the additive valuative structure
\[
(\mathbb Q_p,+,<_{p}),
\qquad
a<_{p}b\ \Longleftrightarrow\ v_p(a)<v_p(b).
\]
The formula
\[
v_p(x_1-y_1)<v_p(x_2-y_2),
\]
with \(x=(x_1,x_2)\) and \(y=(y_1,y_2)\), is not a \((k,1)\)-semiequation for any finite \(k\), as shown in Example~\ref{ex:not-k-wise-valuative}; it remains open whether this formula is equivalent to a Boolean combination of \((k,1)\)-semiequations. More generally, Chernikov and Mennen observed that the additive valuative theory of \(\mathbb C_p\), whose value group is \(\mathbb Q\), is not \((k,1)\)-semiequational for any finite \(k\). Valued vector spaces therefore provide a natural test case for notions of linearity lying beyond the semiequational framework.

We study these structures through semilinear relations, namely finite Boolean combinations of valuative affine-linear comparisons
\[
v(f(x))\ \square\ v(g(x)),
\]
where \(f\) and \(g\) are affine linear and
\(\square\in\{<,\leq,>,\geq\}\). Section~\ref{sect-valued-vector-prelim} proves that every relation definable in the additive affine-valuative structures on \(\mathbb Q_p\) and \(\mathbb C_p\) is semilinear. We establish both absolute and relative Zarankiewicz bounds for this class. For finite sets \(B_1,\ldots,B_r\), write
\[
\delta_{r-1}^{\,r}(B_1\times\cdots\times B_r)
:=
\sum_{j=1}^{r}\prod_{\ell\neq j}|B_\ell|.
\]

\setcounter{manualthm}{0}
\renewcommand{\themanualthm}{\Alph{manualthm}}

\begin{manualthm}\label{thm:main-absolute}
Let $(V,v)$ be a valued $K$-vector space, let $t\geq2$, and fix a description-complexity bound $(\rho,s)$.
\begin{enumerate}
 \item\label{thm:main-absolute-hypergraph}
 For every $r\geq2$, finite sets $B_j\subseteq V^{d_j}$, and every $K_{t,\ldots,t}$-free semilinear relation
 \[
  E\subseteq B_1\times\cdots\times B_r
 \]
 of description complexity at most $(\rho,s)$, there are constants $C=C(r,t,\rho,s)$ and $c=c(r,s)$ such that
 \[
  |E|
  \leq
  C\,\delta_{r-1}^{\,r}(B_1\times\cdots\times B_r)
  \bigl(1+\log(2+\delta_{r-1}^{\,r})\bigr)^c.
 \]
 Consequently, if $n=\sum_j|B_j|$, then
 \[
  |E|=O_{r,t,\rho,s}\!\left(n^{r-1}(\log n)^c\right).
 \]
 \item\label{thm:main-absolute-graph}
 When $r=2$, the arbitrary-trace argument of Section~\ref{sect-upperbound} gives the sharper explicit estimate
 \[
  |E|=O_{\rho,s,t}\!\left(n(1+\log n)^{2s}\right).
 \]
 This improves the logarithmic exponent supplied by the general split-grid induction when specialized to graphs.
 \item\label{thm:main-absolute-lower}
 If $V=K$ is a valued field with infinite value group, then for arbitrarily large $n$ there is a $K_{2,2}$-free semilinear graph with $n$ vertices, point side in $K^2$, parameter side in $K^4$, description complexity $(1,4)$, and
 \[
  \Omega\!\left(n\frac{\log n}{\log\log n}\right)
 \]
 edges.
\end{enumerate}
\end{manualthm}

Part~\eqref{thm:main-absolute-hypergraph} is Theorem~\ref{thm:hypergraph-semilinear}, part~\eqref{thm:main-absolute-graph} is Theorem~\ref{prop: upper-bound}, and part~\eqref{thm:main-absolute-lower} is Theorem~\ref{thm:valuative-lower-bound}. By Corollary~\ref{cor:qe-semilinear}, part~\eqref{thm:main-absolute-hypergraph} applies to every definable relation in the additive affine-valuative structures on $\mathbb Q_p$ and $\mathbb C_p$.

Furthermore, Proposition~\ref{prop:laminar5} shows that after decomposing a valuative comparison into constant-valuation connected components and then taking coordinate projections, laminar families emerge on both sides. This local projection-laminarity illustrates how natural local or quotient operations can reveal the structural geometry that suggests almost-linear Zarankiewicz bounds even when the original moving-radius family is not $k$-wise laminar for any finite $k$.

To state the relative theorem, fix $r\geq2$ and finite sets $B_j\subseteq V^{d_j}$, and let
\[
 A\subseteq B_1\times\cdots\times B_{r-1},
 \qquad
 I\subseteq B_r.
\]
For $A_0\subseteq A$ and $I_0\subseteq I$, put
\[
 \sigma_{r-2}(A_0)
 :=
 \sum_{j=1}^{r-1}
 \left|\pi_{[r-1]\setminus\{j\}}(A_0)\right|,
 \qquad
 \Delta_{\sigma}(A_0,I_0)
 :=
 |A_0|+|I_0|\sigma_{r-2}(A_0),
\]
with the convention $\sigma_0(A_0)=1$ for $A_0\neq\varnothing$ and $0$ otherwise. We write $w(S)=|S|$ for the number of incidences in a relation $S$, and set
\[
 \Lambda
 :=
 1+\left\lceil\log_2\max\{2,\Delta_{\sigma}(A,I)\}\right\rceil.
\]

\begin{manualthm}\label{thm:main-relative}
Let $E_0\subseteq A\times I$ satisfy
\[
 w\bigl(E_0\cap(A_0\times I_0)\bigr)
 \leq
 \alpha_0\Delta_{\sigma}(A_0,I_0)\Lambda^{\ell_0}
\]
for every induced restriction that is $K_{t,\ldots,t}$-free. Let $E$ be obtained from $E_0$ by intersecting with $s$ affine valuative comparisons. Then every induced $K_{t,\ldots,t}$-free restriction of $E$ satisfies
\[
 w\bigl(E\cap(A_0\times I_0)\bigr)
 \leq
 C_s\Delta_{\sigma}(A_0,I_0)\Lambda^{\ell_0+2s},
\]
where $C_s$ depends only on $r,t,s$, and $\alpha_0$.
\end{manualthm}

Theorem~\ref{thm:main-relative} is Theorem~\ref{thm:upper-bound-r-conditional}. It is not subsumed by Theorem~\ref{thm:main-absolute}. The absolute theorem applies to semilinear relations of bounded description complexity in valued vector spaces and gives an immediately usable conclusion for that class. The relative theorem is a closure principle over an arbitrary background relation with a hereditary profile; the background need not itself be semilinear. Conversely, the relative theorem does not supply its own base profile. For $r=2$, the complete relation has the required hereditary linear profile, and Theorem~\ref{thm:main-relative} yields the sharper graph bound in Theorem~\ref{thm:main-absolute}\eqref{thm:main-absolute-graph}. For $r\geq 3$ the absolute base is supplied instead by the split-grid induction of Section~\ref{sect-hypergraph}. Thus the two forms have different logical content: the absolute theorem describes the full semilinear class, while the relative theorem isolates the stability of almost-linear profiles under moving-radius valuative operations.

The logarithmic lower bound shows that no general linear estimate can replace all of these upper bounds, even for a fixed low-complexity bipartite relation.

\subsection*{Organization of the paper}
Section~\ref{sect-valued-vector-prelim} introduces valued vector spaces, semilinear description complexity, and quantifier elimination for the additive affine-valuative structures on $\mathbb Q_p$ and $\mathbb C_p$. Section~\ref{sect-2} proves rectangular propagation and the two projection-laminarity theorem for constant-valuation components. Section~\ref{sect-laminar} establishes the arbitrary-trace cost of adding a laminar family. Section~\ref{sect-upperbound} combines that result with finite ultrametric ball trees to prove the relative moving-radius theorem and the sharper absolute bipartite bound. Section~\ref{sect-hypergraph} proves the absolute bound in all arities by a separate split-grid induction that keeps every lower-arity trace semilinear of bounded complexity. Section~\ref{sect: lower_bound} constructs the point--box examples giving the superlinear lower bound.

\subsection*{Acknowledgements}
The authors are grateful to Erik Walsberg and Artem Chernikov for helpful discussions related to this work. Artificial intelligence tools were used for writing and editing assistance in the preparation of this paper; the authors are responsible for the final content.
\section{Preliminaries on valued vector spaces}
\label{sect-valued-vector-prelim}
\label{sect-valued-field-prelim}

\providecommand{\Sub}{\operatorname{Sub}}
\providecommand{\Cov}{\operatorname{Cov}}

Let $(K,v_K)$ be a valued field with value group $\Gamma_K$. A \emph{valued $K$-vector space} is a $K$-vector space $V$, an ordered abelian group $\Gamma$ containing $\Gamma_K$, and a map
\[
 v:V\longrightarrow\Gamma\cup\{\infty\}
\]
such that, for all $x,y\in V$ and $\lambda\in K$,
\[
 v(x)=\infty\Longleftrightarrow x=0,
 \qquad
 v(x+y)\geq\min\{v(x),v(y)\},
 \qquad
 v(\lambda x)=v_K(\lambda)+v(x).
\]
We use the conventions $v(0)=\infty$ and $\gamma+\infty=\infty$. The principal examples are $V=K$ for $K=\mathbb Q_p$ or $K=\mathbb C_p$, viewed as one-dimensional valued vector spaces over themselves.

For $a,r\in V$, put
\[
 B_{\geq}(a,r):=\{x\in V:v(x-a)\geq v(r)\},
 \qquad
 B_{>}(a,r):=\{x\in V:v(x-a)>v(r)\}.
\]
Thus $B_{\geq}(a,0)=\{a\}$ and $B_{>}(a,0)=\varnothing$. The vector $r$ records the radius through its valuation; this convention permits radii to be affine terms.

\begin{defn}\label{defn: description-complexity}
A set $X\subseteq V^d$ is \emph{semilinear of description complexity $(\rho,s)$} if it is a union of at most $\rho$ sets, each defined by a conjunction of at most $s$ conditions
\[
 v(f_j(x))\ \square_j\ v(g_j(x)),
 \qquad
 \square_j\in\{<,\leq,>,\geq\},
\]
where $f_j,g_j:V^d\to V$ are affine $K$-linear maps. A relation on a product of finite powers of $V$ is semilinear of description complexity $(\rho,s)$ when its graph is such a set in the corresponding Cartesian power.
\end{defn}

For $K=\mathbb Q_p$ or $K=\mathbb C_p$, let
\[
 L_K:=\{0,+,-,(\lambda\cdot)_{\lambda\in K},<_{v}\},
\]
where $a<_{v}b$ abbreviates $v(a)<v(b)$. The value group is not a separate sort. We prove quantifier elimination for the one-sorted structure with universe $K$ in this language.

\begin{lem}\label{lem:qe-linear-normalization}
Every $L_K$-literal in one variable $x$, with parameters $y$, is equivalent to a Boolean combination of conditions
\[
 x\in B_{\geq}(a(y),r(y))
 \quad\text{and}\quad
 x\in B_{>}(a(y),r(y)),
\]
where $a(y)$ and $r(y)$ are affine $K$-linear terms.
\end{lem}

\begin{proof}
Every term in $x$ has the form $\lambda x+a(y)$. Equalities and their negations are singleton conditions and their complements. It is therefore enough to treat
\[
 v(\lambda x+a)<v(\mu x+b).
\]
The cases in which $\lambda=0$ or $\mu=0$ are immediately a ball condition or its complement. Assume $\lambda\mu\neq0$, and put
\[
 c:=-a/\lambda,
 \qquad
 d:=-b/\mu,
 \qquad
 z:=x-d,
 \qquad
 e:=c-d,
 \qquad
 \eta:=\mu/\lambda.
\]
The comparison is
\[
 v(z-e)<v(\eta z).
\]
Write $\delta:=v(\eta)$. If $\delta>0$, then
\[
 v(z-e)<v(\eta z)
 \quad\Longleftrightarrow\quad
 v(z-e)<v(\eta e).
\]
Indeed, if $v(z)<v(e)$, then $v(z-e)=v(z)$ and both inequalities hold; if $v(z)>v(e)$, then $v(z-e)=v(e)$ and both inequalities hold; if $v(z)=v(e)$, the two inequalities are identical. If $\delta=0$, then
\[
 v(z-e)<v(z)
 \quad\Longleftrightarrow\quad
 v(e)<v(z).
\]
For the forward implication, the uniquely attained minimum in $e=z-(z-e)$ gives $v(e)=v(z-e)$; the reverse implication follows from $v(z-e)=v(e)$. If $\delta<0$, then
\[
 v(z-e)<v(\eta z)
 \quad\Longleftrightarrow\quad
 v(e/\eta)<v(z).
\]
For the forward implication, $v(z-e)<v(z)+\delta<v(z)$, hence $v(e)=v(z-e)$; the converse follows from $v(e)<v(z)+\delta<v(z)$ and therefore $v(z-e)=v(e)$. Each resulting condition is a ball condition or its complement. The other comparison symbols follow by negation and interchange of the two sides.
\end{proof}

For $\epsilon,\delta\in\{\geq,>\}$, write $B^{\epsilon}(a,r)$ for the corresponding open or closed ball. After separating the empty-ball cases, inclusion of nonempty balls is quantifier-free definable. In a densely ordered value group one has
\begin{equation}\label{eq:ball-inclusion-dense}
 B^{\epsilon}(a,r)\subseteq B^{\delta}(b,s)
\end{equation}
if and only if $a\in B^{\delta}(b,s)$ and either $v(r)>v(s)$, or $v(r)=v(s)$ and at least one of the following holds: the source is open or the target is closed. We denote this quantifier-free condition by $\Sub(B^{\epsilon}(a,r),B^{\delta}(b,s))$.

\begin{lem}\label{lem:cp-no-finite-cover}
No nonempty open or closed ball in $\mathbb C_p$ is a finite union of proper subballs.
\end{lem}

\begin{proof}
The value group of $\mathbb C_p$ is $\mathbb Q$, and its residue field is infinite. Let $B$ be a nonempty ball.

Suppose first that $B=a+\{x:v(x)\geq\gamma\}$ is closed. Every proper subball of $B$ is contained in one coset of $\{x:v(x)>\gamma\}$ inside $B$. The quotient
\[
 \{x:v(x)\geq\gamma\}/\{x:v(x)>\gamma\}
\]
is a one-dimensional vector space over the residue field and is therefore infinite. Finitely many proper subballs cannot meet all of its cosets.

Suppose that $B=a+\{x:v(x)>\gamma\}$ is open, and let $C_1,\ldots,C_n$ be proper subballs. For each $j$, choose $c_j\in C_j$ and $x_j\in B\setminus C_j$. Then
\[
 \delta_j:=v(x_j-c_j)>\gamma,
\]
and ultrametricity gives
\[
 C_j\subseteq D_j:=\{x:v(x-c_j)\geq\delta_j\}\subsetneq B.
\]
Choose $\delta\in\mathbb Q$ with
\[
 \gamma<\delta<\min_j\delta_j.
\]
Every $D_j$ lies in one coset modulo $\{x:v(x)\geq\delta\}$. On the other hand, $B$ contains infinitely many such cosets: choose elements whose valuations are distinct values in $(\gamma,\delta)$. Hence the $D_j$, and therefore the $C_j$, do not cover $B$.
\end{proof}

\begin{thm}\label{thm:qe-cp}
The structure
\[
 (\mathbb C_p,+,<_{v},(\lambda\cdot)_{\lambda\in\mathbb C_p})
\]
has quantifier elimination.
\end{thm}

\begin{proof}
It suffices to eliminate one existential quantifier from a conjunction of literals. By Lemma~\ref{lem:qe-linear-normalization}, followed by a finite disjunctive-normal-form expansion and a case split for empty open balls, the conjunction has the form
\[
 \psi(y)\wedge
 \bigwedge_{i=1}^{m}x\in P_i(y)
 \wedge
 \bigwedge_{j=1}^{n}x\notin N_j(y),
\]
where $\psi$ is quantifier-free and contains no $x$, and the $P_i,N_j$ are nonempty open or closed balls.

If $m=0$, the existential formula is equivalent to $\psi(y)$, because a finite union of balls is bounded and cannot cover $\mathbb C_p$. Assume $m>0$. A finite family of ultrametric balls has nonempty intersection precisely when one member is contained in all the others, and then the intersection equals that member. By Lemma~\ref{lem:cp-no-finite-cover}, a nonempty ball $P_i$ is not covered by $N_1,\ldots,N_n$ unless it is contained in one of them. Therefore the existential formula is equivalent to
\[
 \psi(y)\wedge
 \bigvee_{i=1}^{m}
 \left(
  \bigwedge_{k=1}^{m}\Sub(P_i,P_k)
  \wedge
  \bigwedge_{j=1}^{n}\neg\Sub(P_i,N_j)
 \right).
\]
This is quantifier-free by \eqref{eq:ball-inclusion-dense}. Iterating the one-variable elimination proves the theorem.
\end{proof}

We now treat $\mathbb Q_p$, normalizing $v(p)=1$. Put
\[
 B(a,r):=B_{\geq}(a,r).
\]
For $r\neq0$,
\[
 B_{>}(a,r)=B(a,pr),
\]
whereas $B_{>}(a,0)=\varnothing$. Thus every ball condition is reduced, after a parameter case split, to a closed-ball condition. For closed balls,
\[
 B(a,r)\subseteq B(b,s)
\]
if and only if $a\in B(b,s)$ and $v(r)\geq v(s)$; this is quantifier-free in $L_{\mathbb Q_p}$.

For a finite tuple $\bar C=(C_1,\ldots,C_n)$ of closed balls, define
\[
 \Cov^0_{\bar C}(B):=\bigvee_{j=1}^{n}\Sub(B,C_j).
\]
If $B=B(a,r)$, define recursively
\[
 \Cov^{d+1}_{\bar C}(B(a,r))
 :=
 \Cov^0_{\bar C}(B(a,r))
 \vee
 \left(
  r\neq0
  \wedge
  \bigwedge_{\ell=0}^{p-1}
  \Cov^d_{\bar C}(B(a+\ell r,pr))
 \right).
\]

\begin{lem}[Finite covers of $p$-adic balls]\label{lem:qe-discrete-cover}
For closed balls $C_1,\ldots,C_n$ in $\mathbb Q_p$ and a closed ball $B$,
\[
 \Cov^n_{\bar C}(B)
 \quad\Longleftrightarrow\quad
 B\subseteq\bigcup_{j=1}^{n}C_j.
\]
For $n=0$, both sides are false when $B$ is nonempty.
\end{lem}

\begin{proof}
The implication from left to right follows by induction on the displayed recursion, using
\[
 B(a,r)=\bigsqcup_{\ell=0}^{p-1}B(a+\ell r,pr)
 \qquad(r\neq0).
\]
For the converse, induct on $n$. The case $n=1$ follows from laminarity of balls. Suppose $n>1$ and that $B$ is covered by $C_1,\ldots,C_n$. If some $C_j$ contains $B$, then $\Cov^0_{\bar C}(B)$ holds. Otherwise $B$ is non-singleton, and every $C_j$ meeting $B$ lies in a unique immediate subball of $B$. The immediate subballs are all nonempty and are covered by pairwise disjoint subcollections of the $C_j$. Since there are at least two immediate subballs, each subcollection has cardinality at most $n-1$. The induction hypothesis, applied to each immediate subball and its relevant subcollection, implies the corresponding formula $\Cov^{n-1}_{\bar C}$; adding unused balls does not destroy the formula. Hence $\Cov^n_{\bar C}(B)$ holds.
\end{proof}

\begin{thm}\label{thm:qe-qp}
The structure
\[
 (\mathbb Q_p,+,<_{v},(\lambda\cdot)_{\lambda\in\mathbb Q_p})
\]
has quantifier elimination.
\end{thm}

\begin{proof}
As in the proof of Theorem~\ref{thm:qe-cp}, after linear normalization, conversion of open balls to closed balls, case splitting, and disjunctive normal form, it is enough to eliminate $x$ from
\[
 \psi(y)\wedge
 \bigwedge_{i=1}^{m}x\in P_i(y)
 \wedge
 \bigwedge_{j=1}^{n}x\notin N_j(y),
\]
where all $P_i,N_j$ are closed balls. If $m=0$, the formula is equivalent to $\psi(y)$ because a finite union of closed balls is bounded. If $m>0$, the positive intersection, when nonempty, equals one of the $P_i$. Lemma~\ref{lem:qe-discrete-cover} therefore gives the quantifier-free equivalent
\[
 \psi(y)\wedge
 \bigvee_{i=1}^{m}
 \left(
  \bigwedge_{k=1}^{m}\Sub(P_i,P_k)
  \wedge
  \neg\Cov^n_{(N_1,\ldots,N_n)}(P_i)
 \right).
\]
Iterating eliminates all quantifiers.
\end{proof}

\begin{cor}\label{cor:qe-semilinear}
Every definable subset of $K^d$ in the additive affine-valuative language, for $K=\mathbb Q_p$ or $K=\mathbb C_p$, is semilinear in the sense of Definition~\ref{defn: description-complexity}.
\end{cor}

\begin{proof}
By Theorems~\ref{thm:qe-cp} and~\ref{thm:qe-qp}, every definable set is quantifier-free definable. Atomic formulas are affine equalities and valuation comparisons between affine terms. An equality $h(x)=0$ is the comparison
\[
 v(h(x))\geq v(0),
\]
and Boolean combinations of the resulting affine valuative comparisons are finite unions of conjunctive cells of the form in Definition~\ref{defn: description-complexity}.
\end{proof}

\begin{rem}
The quantifier-elimination statements above concern the one-dimensional structures $K$ over themselves in languages containing all scalar multiplications. The absolute and relative counting theorems are formulated for semilinear relations in arbitrary valued vector spaces and do not require quantifier elimination there. Corollary~\ref{cor:qe-semilinear} is used only to pass from semilinear to arbitrary definable relations in the additive structures on $\mathbb Q_p$ and $\mathbb C_p$.
\end{rem}
\section{Laminarity in valued vector spaces}
\label{sect-2}

The moving-radius geometry considered later is not itself laminar.  The first example records this obstruction.  We then isolate a projection-laminar structure carried by connected components of constant-valuation graphs.

\begin{exmp}\label{ex:not-k-wise-valuative}
Let
\[
 \phi(x_1,x_2;y_1,y_2)
 :=
 v(x_1-y_1)<v(x_2-y_2)
\]
in \((\mathbb Q_p,+,<_{v})\), with
\(x=(x_1,x_2)\) and \(y=(y_1,y_2)\).  Normalize \(v(p)=1\), and for \(j\geq1\) put
\[
 b_j:=(p^{j-1},p^j),
 \qquad
 F_j:=\phi(\mathbb Q_p^2;b_j).
\]
Then \((0,0)\in F_j\) for every \(j\).  If \(j<k\), then
\[
 (p^{k-1},p^j)\in F_j\setminus F_k,
 \qquad
 (p^{j-1},p^k)\in F_k\setminus F_j.
\]
Hence, for every \(N\), the sets \(F_1,\ldots,F_N\) have nonempty common intersection and are pairwise incomparable.  The family of \(x\)-fibres of \(\phi\) is therefore not \(k\)-wise laminar for any finite \(k\).
\end{exmp}

Let \((V,v)\) be a valued \(K\)-vector space with value group \(\Gamma\).  Let \(A\subseteq V^m\) and \(B\subseteq V^n\) be finite, and let
\begin{equation}\label{eq:single-valuative-relation}
 E
 :=
 \{(x,y)\in A\times B:
   v(f(x,y))\ \square\ v(g(x,y))\},
 \qquad
 \square\in\{>,\geq\},
\end{equation}
where \(f,g:V^{m+n}\to V\) are affine \(K\)-linear maps.  For \(\gamma\in\Gamma\cup\{\infty\}\), define
\[
 E_{=\gamma}
 :=
 \{(x,y)\in E:v(g(x,y))=\gamma\},
 \qquad
 G_{=\gamma}:=(A,B;E_{=\gamma}).
\]
We consider only connected components of \(G_{=\gamma}\) containing at least one edge.  If \(C\) is such a component, put
\[
 A(C):=C\cap A,
 \qquad
 B(C):=C\cap B.
\]

\begin{lem}[Rectangular propagation]\label{lem:3imp4}
Suppose
\[
 (x,y),\ (x,y'),\ (x',y)\in E,
\]
and
\[
 v(g(x,y))=v(g(x,y'))=\gamma,
 \qquad
 v(g(x',y))=\gamma',
 \qquad
 \gamma>\gamma'.
\]
Then \((x',y')\in E\) and
\[
 v(g(x',y'))=\gamma'.
\]
\end{lem}

\begin{proof}
Every affine map \(h:V^{m+n}\to V\) satisfies the rectangular identity
\begin{equation}\label{eq:affine-rectangle-identity}
 h(x',y')
 =
 h(x',y)+h(x,y')-h(x,y).
\end{equation}
Apply this identity to \(g\).  Among the three summands on the right-hand side, the first has valuation \(\gamma'\), while the other two have valuation \(\gamma>\gamma'\).  The ultrametric inequality with a uniquely attained minimum therefore gives
\[
 v(g(x',y'))=\gamma'.
\]

Apply \eqref{eq:affine-rectangle-identity} to \(f\).  If \(\square\) is \(\geq\), then
\[
 v(f(x',y))\geq\gamma',
 \qquad
 v(f(x,y')),v(f(x,y))\geq\gamma,
\]
so \(v(f(x',y'))\geq\gamma'\).  If \(\square\) is \(>\), then all three displayed valuations are strictly greater than \(\gamma'\), and consequently \(v(f(x',y'))>\gamma'\).  In either case \((x',y')\) satisfies \eqref{eq:single-valuative-relation}.
\end{proof}

\begin{prop}[Two-sided projection laminarity]\label{prop:laminar5}
Let \(\mathcal C\) be the collection of all edge-containing connected components of the graphs \(G_{=\gamma}\), as \(\gamma\) varies.  Then both set systems
\[
 \mathcal A:=\{A(C):C\in\mathcal C\},
 \qquad
 \mathcal B:=\{B(C):C\in\mathcal C\}
\]
are laminar.
\end{prop}

\begin{proof}
We prove the assertion for \(\mathcal B\); the proof for \(\mathcal A\) follows after interchanging the two sides.

Let \(C\) and \(C'\) be components of \(G_{=\gamma}\) and \(G_{=\gamma'}\), respectively, and suppose
\[
 B(C)\cap B(C')\neq\varnothing.
\]
Choose \(b_0\in B(C)\cap B(C')\).

If \(\gamma=\gamma'\), then \(C\) and \(C'\) are connected components of the same graph and share the vertex \(b_0\).  Hence \(C=C'\).

Assume \(\gamma>\gamma'\).  Since \(b_0\in B(C')\), there exists \(a_0\in A(C')\) such that
\[
 (a_0,b_0)\in E_{=\gamma'}.
\]
Let \(b\in B(C)\).  Inside \(C\), choose a path from \(b_0\) to \(b\), written on the \(B\)-side as
\[
 b_0,a_1,b_1,a_2,\ldots,a_s,b_s=b,
\]
where
\[
 (a_j,b_{j-1}),(a_j,b_j)\in E_{=\gamma}
 \qquad (1\leq j\leq s).
\]
We prove inductively that \((a_0,b_j)\in E_{=\gamma'}\).  The assertion holds for \(j=0\).  Given it for \(j-1\), Lemma~\ref{lem:3imp4}, applied to
\[
 (a_j,b_{j-1}),\quad
 (a_j,b_j),\quad
 (a_0,b_{j-1}),
\]
produces \((a_0,b_j)\in E_{=\gamma'}\).  Thus every \(b\in B(C)\) lies in \(C'\), and
\[
 B(C)\subseteq B(C').
\]
The alternative ordering of \(\gamma\) and \(\gamma'\) is symmetric.  Therefore \(\mathcal B\) is laminar.
\end{proof}


\section{The cost of adding a laminar family}
\label{sect-laminar}

Fix integers \(r,t\geq2\).  Let
\[
 X\subseteq V_1\times\cdots\times V_{r-1},
 \qquad
 I\subseteq V_r
\]
be finite.  A relation \(R\subseteq X\times I\) is viewed as an \(r\)-partite hypergraph.  Its weight is
\[
 w(R):=|R|.
\]
It is \(K_{t,\ldots,t}\)-free if there are no sets
\[
 T_j\subseteq V_j\quad(1\leq j\leq r-1),
 \qquad
 J\subseteq I,
\]
all of cardinality \(t\), such that
\[
 T_1\times\cdots\times T_{r-1}\subseteq X
\]
and
\[
 (T_1\times\cdots\times T_{r-1})\times J\subseteq R.
\]

For \(Y\subseteq X\), define
\[
 \sigma_{r-2}(Y)
 :=
 \sum_{j=1}^{r-1}
 \left|
  \pi_{[r-1]\setminus\{j\}}(Y)
 \right|.
\]
When \(r=2\), put \(\sigma_0(Y)=1\) for \(Y\neq\varnothing\) and \(\sigma_0(\varnothing)=0\).  We use the arbitrary-trace measure
\begin{equation}\label{eq:arbitrary-trace-measure}
 \Delta_{\sigma}(Y,J)
 :=
 |Y|+|J|\sigma_{r-2}(Y)
\end{equation}
and the ambient logarithmic parameter
\begin{equation}\label{eq:arbitrary-trace-log}
 \Lambda_{\sigma}(X,I)
 :=
 1+\left\lceil
  \log_2\max\{2,\Delta_{\sigma}(X,I)\}
 \right\rceil.
\end{equation}

\begin{lem}[Interval representation of a finite laminar family]
\label{lem:laminar-interval-order}
Let \(\mathcal D\) be a finite laminar family of subsets of a finite set \(Y\).  There is a linear order on \(Y\) in which every member of \(\mathcal D\) is an interval.
\end{lem}

\begin{proof}
Discard the empty set and identify repeated members.  Order the distinct members by inclusion.  Their Hasse diagram is a forest because the family is laminar.  For each node \(D\), first list recursively the elements belonging to the subtrees of its children, in an arbitrary order of the children, and then list the elements of
\[
 D\setminus\bigcup_{C\text{ child of }D}C.
\]
List the root subtrees consecutively and place the elements of \(Y\) contained in no member of \(\mathcal D\) last.  The elements belonging to the subtree rooted at \(D\) form one consecutive block, and this block is precisely \(D\).  Repeated members and the empty set are therefore intervals as well.
\end{proof}

\begin{lem}[Canonical decomposition of an interval]
\label{lem:canonical-interval-decomposition}
Let \(Y\) be a finite linearly ordered set of cardinality \(N\geq1\).  There is a balanced binary interval tree \(\mathcal T_Y\) of depth at most
\[
 1+\lceil\log_2 N\rceil
\]
with the following property.  Every interval \(J\subseteq Y\) is a disjoint union of at most
\[
 2+2\lceil\log_2 N\rceil
\]
node intervals of \(\mathcal T_Y\).
\end{lem}

\begin{proof}
Construct \(\mathcal T_Y\) by splitting every non-singleton interval into its balanced left and right halves.  For a given interval \(J\), take the maximal nodes whose intervals are contained in \(J\).  These intervals are disjoint and cover \(J\).  At each depth, at most two nodes meeting \(J\) fail to be contained in it, namely the nodes meeting the two boundary points of \(J\).  Consequently at most two maximal contained nodes are created at each depth, which gives the stated bound.
\end{proof}

\begin{prop}[Relative laminar extension]
\label{prop:hypergraph_main_t}
Let \(R'\subseteq X\times I\), let \(G\geq1\), and suppose that every induced restriction
\[
 R'\cap(Y\times J)
\]
which is \(K_{t,\ldots,t}\)-free satisfies
\begin{equation}\label{eq:base-hereditary-profile}
 w\bigl(R'\cap(Y\times J)\bigr)
 \leq
 \beta\,\Delta_{\sigma}(Y,J)\,G.
\end{equation}
Assume that one of the following holds.
\begin{enumerate}
 \item For each \(i\in I\), a set \(D_i\subseteq X\) is given, and \(\{D_i:i\in I\}\) is laminar.  Put
 \[
  R:=R'\cap\{(x,i):x\in D_i\}.
 \]
 \item For each \(x\in X\), a set \(D_x\subseteq I\) is given, and \(\{D_x:x\in X\}\) is laminar.  Put
 \[
  R:=R'\cap\{(x,i):i\in D_x\}.
 \]
\end{enumerate}
If \(R\) is \(K_{t,\ldots,t}\)-free, then
\begin{equation}\label{eq:relative-laminar-bound}
 w(R)
 \leq
 C\beta\,\Delta_{\sigma}(X,I)\,
 \Lambda_{\sigma}(X,I)\,G,
\end{equation}
where \(C\) is an absolute constant.
\end{prop}

\begin{proof}
If \(X=\varnothing\) or \(I=\varnothing\), then \(R=\varnothing\) and the conclusion is immediate.  Assume henceforth that both sets are nonempty.

Assume first that the family is laminar on \(X\).  By Lemma~\ref{lem:laminar-interval-order}, order \(X\) so that every \(D_i\) is an interval.  Let \(\mathcal T_X\) be the binary interval tree from Lemma~\ref{lem:canonical-interval-decomposition}.  For each \(i\in I\), let \(\mathcal K_i\) be the canonical node intervals whose disjoint union is \(D_i\), and for a node \(u\) put
\[
 C_u:=\text{the interval represented by }u,
 \qquad
 I_u:=\{i\in I:C_u\in\mathcal K_i\}.
\]
Every edge \((x,i)\in R\) belongs to a unique block \(C_u\times I_u\).  Moreover, the laminar condition is automatic on this block, so
\[
 R\cap(C_u\times I_u)
 =
 R'\cap(C_u\times I_u).
\]
The left-hand side is an induced subrelation of \(R\), hence is \(K_{t,\ldots,t}\)-free.  Applying \eqref{eq:base-hereditary-profile} and summing gives
\[
 w(R)
 \leq
 \beta G
 \sum_{u\in\mathcal T_X}
 \Delta_{\sigma}(C_u,I_u).
\]
At each depth of \(\mathcal T_X\), the intervals \(C_u\) are disjoint, whence
\[
 \sum_u|C_u|
 \leq
 |X|\bigl(1+\lceil\log_2|X|\rceil\bigr).
\]
Also,
\[
 \sum_u|I_u|\sigma_{r-2}(C_u)
 \leq
 \sigma_{r-2}(X)
 \sum_{i\in I}|\mathcal K_i|
 \leq
 C|I|\sigma_{r-2}(X)
 \bigl(1+\lceil\log_2|X|\rceil\bigr).
\]
Since \(|X|\leq\Delta_{\sigma}(X,I)\), these inequalities imply \eqref{eq:relative-laminar-bound}.

Assume now that the family is laminar on \(I\).  Order \(I\) by Lemma~\ref{lem:laminar-interval-order}, use the binary interval tree \(\mathcal T_I\), and decompose every \(D_x\) canonically.  For a node interval \(J_u\subseteq I\), put
\[
 X_u:=\{x\in X:J_u\text{ occurs in the canonical decomposition of }D_x\}.
\]
The same argument gives
\[
 w(R)
 \leq
 \beta G
 \sum_{u\in\mathcal T_I}
 \bigl(|X_u|+|J_u|\sigma_{r-2}(X_u)\bigr).
\]
Each \(x\) belongs to at most
\(C(1+\log|I|)\) sets \(X_u\), and at each depth the intervals \(J_u\) are disjoint.  Since \(\sigma_{r-2}(X_u)\leq\sigma_{r-2}(X)\),
\[
 \sum_u|X_u|
 \leq
 C|X|(1+\log|I|),
\]
and
\[
 \sum_u|J_u|\sigma_{r-2}(X_u)
 \leq
 C|I|\sigma_{r-2}(X)(1+\log|I|).
\]
This again yields \eqref{eq:relative-laminar-bound}.
\end{proof}

\section{Relative moving-radius extensions and bipartite bounds}
\label{sect-upperbound}

The preceding proposition gives a relative counting principle for arbitrary induced traces.  We now combine it with the finite ultrametric ball tree to control one moving-radius comparison.  This relative theorem is independent of the subsequent absolute split-grid induction; its principal absolute consequence is the bipartite estimate at the end of the section.

Throughout this section, \(B_1,\ldots,B_r\) are finite sets, \(A\subseteq B_1\times\cdots\times B_{r-1}\), and \(I\subseteq B_r\).  We write
\[
 \Lambda:=\Lambda_{\sigma}(A,I).
\]
A relation \(E'\subseteq A\times I\) has hereditary profile \((\beta,\ell)\) if every induced restriction \(E'\cap(A_0\times I_0)\) which is \(K_{t,\ldots,t}\)-free satisfies
\begin{equation}\label{eq:power-hereditary-profile}
 w\bigl(E'\cap(A_0\times I_0)\bigr)
 \leq
 \beta\,\Delta_{\sigma}(A_0,I_0)\,\Lambda^{\ell}.
\end{equation}

\begin{lem}[Finite ultrametric ball-tree geometry]
\label{lem:ball_tree_geom}
Let \(S\subseteq V\) be finite and nonempty.  There is a rooted tree \(\mathcal T(S)\) whose leaves are the singletons \(\{z\}\), \(z\in S\), and whose nodes are traces on \(S\) of ultrametric balls.  If \(u\) is a non-leaf node, then there is a value \(\gamma_u\) such that its children partition \(u\), and points lying in distinct children have mutual valuation \(\gamma_u\).  If \(z\notin u\) and \(z',z''\in u\), then
\[
 v(z-z')=v(z-z'').
\]
\end{lem}

\begin{proof}
If \(|S|=1\), take the one-node tree.  Otherwise let
\[
 \gamma_S
 :=
 \min\{v(z-z'):z,z'\in S,\ z\neq z'\}.
\]
The relation \(z\sim z'\) if and only if \(v(z-z')>\gamma_S\) is an equivalence relation with at least two classes.  These classes are traces of open balls, and points in distinct classes have mutual valuation \(\gamma_S\).  Apply the same construction recursively to each class.  The cardinality decreases at every nontrivial step, so the process terminates.

For the final assertion, let \(u\) be represented by a ball of radius level \(\gamma\).  Then \(v(z'-z'')\geq\gamma\), whereas \(v(z-z')<\gamma\).  The uniquely attained minimum in
\[
 z-z''=(z-z')+(z'-z'')
\]
gives \(v(z-z'')=v(z-z')\).
\end{proof}

\begin{lem}[Fixed-radius localization]
\label{lem:fixed_radius_r}
Let \(E'\subseteq A\times I\) have hereditary profile \((\beta,\ell)\).  Let
\[
 D_\alpha
 :=
 \{(a,i):v(F(a)+M(i))\ \square\ \alpha\},
 \qquad
 \square\in\{>,\geq\},
\]
where \(F\) and \(M\) are affine and \(\alpha\in\Gamma\cup\{\infty\}\).  If
\[
 E:=E'\cap D_\alpha
\]
is \(K_{t,\ldots,t}\)-free, then
\[
 w(E)
 \leq
 \beta\,\Delta_{\sigma}(A,I)\,\Lambda^{\ell}.
\]
\end{lem}

\begin{proof}
For finite \(\alpha\), put
\[
 H_{\geq\alpha}:=\{z:v(z)\geq\alpha\},
 \qquad
 H_{>\alpha}:=\{z:v(z)>\alpha\}.
\]
If \(\square\) is \(\geq\), then \((a,i)\in D_\alpha\) precisely when the images of \(F(a)\) and \(-M(i)\) agree modulo \(H_{\geq\alpha}\); for \(>\), use \(H_{>\alpha}\).  Partition \(A\) and \(I\) by the corresponding quotient labels.  The relation \(D_\alpha\) is the disjoint union of the matched-label products \(A_\xi\times I_\xi\).  On each such product,
\[
 E\cap(A_\xi\times I_\xi)
 =
 E'\cap(A_\xi\times I_\xi),
\]
and this restriction is \(K_{t,\ldots,t}\)-free.  Hence
\[
 w(E)
 \leq
 \beta\Lambda^{\ell}
 \sum_\xi\Delta_{\sigma}(A_\xi,I_\xi).
\]
Since the label classes partition both sides and
\(\sigma_{r-2}(A_\xi)\leq\sigma_{r-2}(A)\),
\[
 \sum_\xi\Delta_{\sigma}(A_\xi,I_\xi)
 \leq
 \Delta_{\sigma}(A,I).
\]

If \(\alpha=\infty\) and \(\square\) is \(>\), the relation is empty.  If \(\alpha=\infty\) and \(\square\) is \(\geq\), it is the equality \(F(a)+M(i)=0\), and the same argument applies with equality labels in \(V\).
\end{proof}

\begin{prop}[Relative moving-radius extension]
\label{prop:main_H_5_relative}
Let \(E'\subseteq A\times I\) have hereditary profile \((\beta,\ell)\), and let
\[
 D
 :=
 \{(a,i):v(f(a,i))\ \square\ v(g(a,i))\},
 \qquad
 \square\in\{>,\geq\},
\]
where \(f,g\) are affine \(K\)-linear maps.  If
\[
 E:=E'\cap D
\]
is \(K_{t,\ldots,t}\)-free, then
\begin{equation}\label{eq:relative-moving-radius-bound}
 w(E)
 \leq
 C\beta\,\Delta_{\sigma}(A,I)\,\Lambda^{\ell+2},
\end{equation}
where \(C\) depends only on \(r\) and \(t\).
\end{prop}

\begin{proof}
Write
\[
 f(a,i)=F(a)+M(i),
 \qquad
 g(a,i)=P(a)-Q(i).
\]
We construct a recursion tree.  At a node carrying finite sets \(A_0\subseteq A\) and \(I_0\subseteq I\), all logarithmic factors are measured with the fixed ambient parameter \(\Lambda\).  Nodes with \(|I_0|\leq2\) are terminal, since
\[
 w\bigl(E\cap(A_0\times I_0)\bigr)
 \leq
 2|A_0|
 \leq
 2\Delta_{\sigma}(A_0,I_0).
\]
Assume \(|I_0|\geq3\), and build the ball tree of
\[
 P(A_0)\cup Q(I_0).
\]
For a node \(u\), write
\[
 A_{0,u}:=\{a\in A_0:P(a)\in u\},
 \qquad
 I_{0,u}:=\{i\in I_0:Q(i)\in u\},
\]
and give \(u\) the weight \(|I_{0,u}|\).

Suppose first that a leaf \(\{z\}\) has weight greater than \(2|I_0|/3\).  Put
\[
 I_{0,z}:=\{i\in I_0:Q(i)=z\}.
\]
On \(A_0\times I_{0,z}\), the fibres of \(D\) on the \(I_0\)-side are
\[
 D_a
 =
 \{i\in I_{0,z}:v(F(a)+M(i))\ \square\ v(P(a)-z)\}.
\]
Each \(D_a\) is the inverse image under \(M\) of an open or closed ball.  As \(a\) varies, these balls are laminar, and so are their inverse images.  Proposition~\ref{prop:hypergraph_main_t}, with \(G=\Lambda^\ell\), gives
\[
 w\bigl(E\cap(A_0\times I_{0,z})\bigr)
 \leq
 C\beta\Delta_{\sigma}(A_0,I_{0,z})\Lambda^{\ell+1}.
\]
Declare this block terminal and continue with the single child
\[
 A_0\times(I_0\setminus I_{0,z}),
\]
whose second-coordinate size is less than \(|I_0|/3\).

Assume next that no leaf has weight greater than \(2|I_0|/3\).  Choose a node \(u\), minimal under inclusion, with \(|I_{0,u}|>2|I_0|/3\).  It is not a leaf, and every child \(c\) of \(u\) satisfies \(|I_{0,c}|\leq2|I_0|/3\).  There is a set \(\mathcal S\) of children such that
\begin{equation}\label{eq:balanced-child-selection}
 \frac{|I_0|}{3}
 \leq
 \sum_{c\in\mathcal S}|I_{0,c}|
 \leq
 \frac{2|I_0|}{3}.
\end{equation}
Indeed, if one child has weight at least \(|I_0|/3\), take that child; otherwise add positive-weight children until the sum first reaches \(|I_0|/3\).

Put
\[
 I_{0,\mathrm{in}}
 :=
 \bigcup_{c\in\mathcal S}I_{0,c},
 \qquad
 I_{0,\mathrm{out}}:=I_0\setminus I_{0,\mathrm{in}},
\]
and define \(A_{0,\mathrm{in}}\) and \(A_{0,\mathrm{out}}\) analogously using the \(P\)-preimages of the selected children.  The two diagonal products
\[
 A_{0,\mathrm{in}}\times I_{0,\mathrm{in}},
 \qquad
 A_{0,\mathrm{out}}\times I_{0,\mathrm{out}}
\]
are the recursive children.

We estimate the crossing products.  First consider
\(A_{0,\mathrm{in}}\times I_{0,\mathrm{out}}\).  If \(i\in I_{0,u}\setminus I_{0,\mathrm{in}}\), then \(P(a)\) and \(Q(i)\) lie in distinct children of \(u\), so
\[
 v(P(a)-Q(i))=\gamma_u.
\]
Lemma~\ref{lem:fixed_radius_r} bounds this block with no additional logarithm.  If \(i\in I_0\setminus I_{0,u}\), choose \(q_u\in u\).  Lemma~\ref{lem:ball_tree_geom} gives
\[
 v(P(a)-Q(i))=v(q_u-Q(i)),
\]
independently of \(a\in A_{0,\mathrm{in}}\).  Thus the fibres on the \(A_0\)-side are inverse images of balls and form a laminar family.  Proposition~\ref{prop:hypergraph_main_t} bounds this block by
\[
 C\beta\,
 \Delta_{\sigma}(A_{0,\mathrm{in}},I_0\setminus I_{0,u})
 \Lambda^{\ell+1}.
\]

The product \(A_{0,\mathrm{out}}\times I_{0,\mathrm{in}}\) is symmetric.  On
\((A_{0,u}\setminus A_{0,\mathrm{in}})\times I_{0,\mathrm{in}}\), the radius is the fixed value \(\gamma_u\).  On
\((A_0\setminus A_{0,u})\times I_{0,\mathrm{in}}\), it depends only on \(a\), and Proposition~\ref{prop:hypergraph_main_t} applies to the laminar family on \(I_0\).

The four crossing estimates have total cost at most
\begin{equation}\label{eq:crossing-charge}
 C\beta\,\Delta_{\sigma}(A_0,I_0)\Lambda^{\ell+1}.
\end{equation}
Indeed, every first-coordinate set occurring in those four blocks is contained in \(A_0\), every second-coordinate set is contained in \(I_0\), and the repeated first- and second-coordinate contributions have uniformly bounded multiplicity.

For the diagonal children,
\begin{equation}\label{eq:diagonal-delta-packing}
 \Delta_{\sigma}(A_{0,\mathrm{in}},I_{0,\mathrm{in}})
 +
 \Delta_{\sigma}(A_{0,\mathrm{out}},I_{0,\mathrm{out}})
 \leq
 \Delta_{\sigma}(A_0,I_0),
\end{equation}
because both coordinate sets are partitioned and
\(\sigma_{r-2}\) is monotone.  By \eqref{eq:balanced-child-selection}, each child has second-coordinate size at most \(2|I_0|/3\).

Every root-to-leaf path of the recursion tree has length \(O(\Lambda)\).  At any fixed depth, the active nodes have pairwise disjoint first-coordinate sets and pairwise disjoint second-coordinate sets; hence repeated use of \eqref{eq:diagonal-delta-packing} gives total active \(\Delta_{\sigma}\)-mass at most \(\Delta_{\sigma}(A,I)\).  Heavy-leaf extraction creates only one active remainder and contributes one terminal charge bounded by the mass of its parent, so it changes this estimate by at most an absolute factor.  Summing \eqref{eq:crossing-charge} over the depths proves \eqref{eq:relative-moving-radius-bound}.
\end{proof}

\begin{thm}[Conditional semilinear extension]
\label{thm:upper-bound-r-conditional}
Let \(E_0\subseteq A\times I\) satisfy
\[
 w\bigl(E_0\cap(A_0\times I_0)\bigr)
 \leq
 \alpha_0\Delta_{\sigma}(A_0,I_0)\Lambda^{\ell_0}
\]
for every induced restriction which is \(K_{t,\ldots,t}\)-free.  Let \(E\) be obtained from \(E_0\) by intersecting with \(s\) affine valuative comparisons.  Then every induced \(K_{t,\ldots,t}\)-free restriction of \(E\) satisfies
\begin{equation}\label{eq:conditional-semilinear-bound}
 w\bigl(E\cap(A_0\times I_0)\bigr)
 \leq
 C_s\Delta_{\sigma}(A_0,I_0)\Lambda^{\ell_0+2s},
\end{equation}
where \(C_s\) depends only on \(s,r,t\), and \(\alpha_0\).
\end{thm}

\begin{proof}
Induct on \(s\).  The case \(s=0\) is the hypothesis.  For the induction step, write \(E=E'\cap D\) and work inside an arbitrary induced restriction of \(E\) which is \(K_{t,\ldots,t}\)-free.  Sign normalization replaces \(<\) and \(\leq\) by \(>\) and \(\geq\) after interchanging the two sides of the comparison.  The induction hypothesis gives the required hereditary profile for \(E'\).  If the right-hand side of the normalized comparison is constant, apply Lemma~\ref{lem:fixed_radius_r}; otherwise apply Proposition~\ref{prop:main_H_5_relative}.  Each comparison increases the exponent by at most two.
\end{proof}

\begin{thm}[Absolute bipartite semilinear bound]
\label{prop: upper-bound}
Let \(A\subseteq V^{d_1}\) and \(B\subseteq V^{d_2}\) be finite, and put
\[
 n:=|A|+|B|,
 \qquad
 L_n:=1+\left\lceil\log_2\max\{2,n\}\right\rceil.
\]
Let \(E\subseteq A\times B\) be a union of at most \(\rho\) conjunctive cells, each defined by at most \(s\) affine valuative comparisons.  If \(E\) is \(K_{t,t}\)-free, then
\begin{equation}\label{eq:absolute-bipartite-bound}
 |E|
 \leq
 C_{\rho,s,t}\,n\,L_n^{2s}.
\end{equation}
\end{thm}

\begin{proof}
It suffices to treat one conjunctive cell, because every cell is a subrelation of \(E\) and is therefore \(K_{t,t}\)-free.  The complete relation \(A\times B\) has the following hereditary profile: if \(A_0\times B_0\) is \(K_{t,t}\)-free, then \(\min\{|A_0|,|B_0|\}<t\), and hence
\[
 |A_0||B_0|
 \leq
 (t-1)(|A_0|+|B_0|).
\]
For \(r=2\), \(\Delta_{\sigma}(A_0,B_0)=|A_0|+|B_0|\) whenever \(A_0\neq\varnothing\).  Apply Theorem~\ref{thm:upper-bound-r-conditional} with \(\ell_0=0\) and then sum over the at most \(\rho\) cells.
\end{proof}

\begin{rem}\label{rem:relative-versus-absolute}
For \(r\geq3\), the complete grouped relation \(A\times I\) does not satisfy the arbitrary-trace hereditary profile required by Theorem~\ref{thm:upper-bound-r-conditional}.  The absolute higher-arity theorem is therefore proved separately by split-grid induction on bounded-complexity semilinear traces in the next section.
\end{rem}

\section{Absolute semilinear hypergraph bounds via split-grid induction}
\label{sect-hypergraph}

This section proves Theorem~\ref{thm:main-absolute}\eqref{thm:main-absolute-hypergraph}.  Let
\[
 A\subseteq B_1\times\cdots\times B_{r-1},
 \qquad I\subseteq B_r,
\]
and write \(A_j:=\pi_j(A)\).  Define
\begin{equation}\label{eq:coordinate-hull-shadow}
 \delta_{r-2}^{\,r-1}(A)
 :=
 \sum_{j=1}^{r-1}
 \prod_{\substack{1\leq \ell\leq r-1\\ \ell\neq j}}
 |A_\ell|.
\end{equation}
For \(r=2\), set \(\delta_0^1(A)=1\) when \(A\neq\varnothing\), and \(\delta_0^1(\varnothing)=0\).  Put
\begin{equation}\label{eq:split-measure-section6}
 \Delta(A,I):=|A|+|I|\delta_{r-2}^{\,r-1}(A),
 \qquad
 L(A,I):=1+\left\lceil\log_2\max\{2,\Delta(A,I)\}\right\rceil.
\end{equation}
If \(A'\subseteq A\), then
\(\delta_{r-2}^{\,r-1}(A')\leq\delta_{r-2}^{\,r-1}(A)\).
Moreover, for the full grid \(A=B_1\times\cdots\times B_{r-1}\),
\begin{equation}\label{eq:full-grid-identity}
 \Delta(A,B_r)
 =
 \delta_{r-1}^{\,r}(B_1\times\cdots\times B_r).
\end{equation}

A \emph{semilinear trace of complexity at most \(\kappa\)} is a semilinear subset of \(B_1\times\cdots\times B_{r-1}\) whose description complexity is bounded by \(\kappa\).  Parameters from the ambient valued vector space are permitted.

\begin{lem}[Finite closure of trace complexity]
\label{lem:trace-bookkeeping}
Fix integers \(\kappa,h\geq0\).  Let \(A\) be a semilinear trace of complexity at most \(\kappa\).  Impose at most \(h\) additional conditions of the following forms:
\begin{enumerate}
 \item membership or non-membership of an affine form in a fixed open or closed ball;
 \item membership or non-membership of a fixed point \(z\in V\) in a parametrized ball
 \[
   C_a=\{x\in V:v(x-c(a))\ \diamond\ v(r(a))\},
   \qquad \diamond\in\{>,\geq\},
 \]
 whose center and radius are affine in \(a\);
 \item equality or inequality of a quotient label modulo one of the additive subgroups
 \(\{x:v(x)\geq\gamma\}\) or \(\{x:v(x)>\gamma\}\), for a fixed finite \(\gamma\).
\end{enumerate}
The resulting set is a disjoint union of at most \(N(\kappa,h)\) semilinear traces, each of description complexity at most \(\kappa'(\kappa,h)\).  These bounds are independent of the finite vertex sets and of the chosen parameters.
\end{lem}

\begin{proof}
Each condition in (1) is an affine valuation comparison or its negation.  For (2), for example,
\[
 z\in C_a
 \quad\Longleftrightarrow\quad
 v(z-c(a))\ \diamond\ v(r(a)).
\]
A quotient-label equality is a fixed-radius valuation comparison, and a quotient-label inequality is its negation.  A Boolean combination of a bounded number of such literals admits a disjoint disjunctive-normal-form expansion with uniformly bounded length and description complexity.
\end{proof}


\begin{lem}[Ultrametric balls and finite ball trees]
\label{lem:finite-ball-tree}
Any two nonempty open or closed balls in \(V\) are either disjoint or comparable by inclusion.  Moreover, every finite nonempty set \(S\subseteq V\) admits a rooted tree \(\mathcal T(S)\) with the following properties:
\begin{enumerate}
 \item the leaves are the singletons \(\{z\}\), \(z\in S\);
 \item every node is the trace on \(S\) of an ambient open or closed ball;
 \item if \(b\) is a non-leaf node, there is a value \(\gamma_b\) such that its children partition \(b\), and points in distinct children have mutual valuation exactly \(\gamma_b\).
\end{enumerate}
\end{lem}

\begin{proof}
The first assertion follows directly from the ultrametric inequality.  For the second, proceed recursively.  If \(S\) is a singleton, it is a leaf.  Otherwise put
\[
 \gamma_S:=\min\{v(x-y):x,y\in S,\ x\neq y\}.
\]
The relation \(x\sim y\) if and only if \(v(x-y)>\gamma_S\) is an equivalence relation with at least two classes.  These classes are the traces on \(S\) of the open balls \(\{z:v(z-x)>\gamma_S\}\), and points in distinct classes have mutual valuation \(\gamma_S\).  Apply the same construction recursively to each class.  Since the cardinality decreases at every nontrivial step, the construction terminates.
\end{proof}

\begin{lem}[Dyadic semigraph decomposition]
\label{lem:dyadic-halfgraph-section6}
Let
\[
 \alpha:U\longrightarrow\{0,\ldots,m\},
 \qquad
 \beta:W\longrightarrow\{0,\ldots,m\}.
\]
There is an integer \(N\geq0\) and integer intervals
\(J_\nu,K_\nu\subseteq\{0,\ldots,m\}\), \(1\leq\nu\leq N\), such that
\begin{equation}\label{eq:dyadic-semigraph-partition}
 \{(u,w)\in U\times W:\alpha(u)\leq\beta(w)\}
 =
 \bigsqcup_{\nu=1}^{N}
 \alpha^{-1}(J_\nu)\times\beta^{-1}(K_\nu).
\end{equation}
Each \(u\in U\) and each \(w\in W\) belongs to at most
\(2+2\lceil\log_2(m+1)\rceil\) rectangles on the corresponding side.  The same conclusion holds for \(<,>,\geq\).
\end{lem}

\begin{proof}
For an integer interval \(J\subseteq\{0,\ldots,m\}\), put
\[
 \mathcal H_{\leq}(J):=\{(j,k)\in J^2:j\leq k\}.
\]
If \(J\) is a singleton, use the rectangle \(J\times J\).  Otherwise write \(J=J_1\sqcup J_2\) as its balanced left and right halves.  Then
\[
 \mathcal H_{\leq}(J)
 =
 \mathcal H_{\leq}(J_1)
 \sqcup
 (J_1\times J_2)
 \sqcup
 \mathcal H_{\leq}(J_2).
\]
Iteration gives a disjoint rectangle partition of \(\mathcal H_{\leq}(\{0,\ldots,m\})\).  A fixed integer follows one branch of the recursion and belongs to at most one cross-rectangle at each depth, together with one terminal singleton rectangle.  The multiplicity is therefore at most
\[
 1+\lceil\log_2(m+1)\rceil
\]
on each coordinate side.  Pullback by \(\alpha\) and \(\beta\) proves \eqref{eq:dyadic-semigraph-partition}.  The remaining comparison symbols follow by reversing the order or shifting one of the two maps.
\end{proof}

\begin{lem}[Recognition of a path interval]
\label{lem:path-interval-recognition}
Let \(\mathcal B\) be a finite laminar family of nonempty balls, and let
\[
 C_1\supsetneq C_2\supsetneq\cdots\supsetneq C_m
\]
be a maximal heavy path in its inclusion forest.  Fix \(1\leq p\leq q\leq m\).  Choose \(z_*\in C_m\).  When \(q<m\), choose \(z_q\in C_q\setminus C_{q+1}\).  When \(C_p\) has a parent \(\widehat C_p\) in the forest, choose \(y_p\in\widehat C_p\setminus C_p\).  For every \(C\in\mathcal B\), the condition
\(C\in\{C_p,\ldots,C_q\}\) is equivalent to the conjunction of
\begin{enumerate}
 \item \(z_*\in C\);
 \item \(z_q\in C\), when \(q<m\);
 \item \(y_p\notin C\), when \(C_p\) has a parent.
\end{enumerate}
\end{lem}

\begin{proof}
A ball in \(\mathcal B\) containing \(z_*\) meets the leaf \(C_m\).  Laminarity and the minimality of \(C_m\) in the forest imply that it lies on the ancestor chain of \(C_m\).  Along this chain, the balls containing \(z_q\) are precisely \(C_q\) and its ancestors, while the balls containing \(y_p\) are precisely the proper ancestors of \(C_p\).  The stated conjunction therefore selects exactly \(C_p,\ldots,C_q\).
\end{proof}


\begin{lem}[Rectangularization of one-sided ball relations]
\label{lem:one-sided-rectangularization}
Let \(A\subseteq B_1\times\cdots\times B_{r-1}\) be a semilinear trace of complexity at most \(\kappa\), let \(I\subseteq B_r\), and write
\(f(a,i)=F(a)+M(i)\) with \(F,M\) affine.  Let \(\triangleright\in\{\geq,>\}\).  Each of the relations
\begin{align}
 D_A&:=\{(a,i):v(f(a,i))\triangleright v(G(a))\},
 \label{eq:one-sided-A}\\
 D_I&:=\{(a,i):v(f(a,i))\triangleright v(H(i))\}
 \label{eq:one-sided-I}
\end{align}
admits a disjoint rectangular decomposition with an integer \(N\geq0\) and rectangles
\(A_\nu\times I_\nu\), \(1\leq\nu\leq N\), such that
\begin{equation}\label{eq:one-sided-rectangularization}
 D_*\cap(A\times I)
 =
 \bigsqcup_{\nu=1}^{N}A_\nu\times I_\nu.
\end{equation}
Every \(A_\nu\) has description complexity bounded in terms of \(\kappa\), and
\begin{equation}\label{eq:one-sided-packing}
 \sum_{\nu=1}^{N}\Delta(A_\nu,I_\nu)
 \leq
 C_\kappa\,\Delta(A,I)L(A,I)^2.
\end{equation}
\end{lem}

\begin{proof}
If \(A=\varnothing\) or \(I=\varnothing\), the assertion is immediate.  Set \(L=L(A,I)\).  We first treat \(D_I\).  For \(i\in I\), define
\[
 B_i:=\{z\in V:v(z+M(i))\triangleright v(H(i))\}.
\]
If \(B_i=\varnothing\), then \(i\) contributes no edge and is discarded.  Otherwise,
\((a,i)\in D_I\) if and only if \(F(a)\in B_i\).
After duplicate balls are merged, the distinct balls \(B_i\) form a finite inclusion forest; the weight of a node is the number of original indices whose balls lie in its subtree.  At every non-leaf node choose a child of maximal weight as the heavy child.  Every light child has at most half the weight of its parent.  Hence the chain of balls containing any fixed point of \(V\) meets at most \(O(L)\) heavy paths.

Fix a heavy path
\[
 \mathcal Q:B_1\supsetneq\cdots\supsetneq B_m.
\]
For an index whose ball is \(B_j\), put \(\operatorname{idx}_{\mathcal Q}(i)=j\).  For \(a\in A\), put
\[
 \operatorname{dep}_{\mathcal Q}(a)
 :=
 \max\{j:F(a)\in B_j\},
\]
with value \(0\) if \(F(a)\notin B_1\).  Then
\[
 F(a)\in B_i
 \quad\Longleftrightarrow\quad
 \operatorname{idx}_{\mathcal Q}(i)
 \leq
 \operatorname{dep}_{\mathcal Q}(a).
\]
Apply Lemma~\ref{lem:dyadic-halfgraph-section6}.  If a first-coordinate interval for the depth is \([p,q]\subseteq[1,m]\), its preimage is
\[
 \{a:F(a)\in B_p\setminus B_{q+1}\},
\]
where the inner ball is omitted when \(q=m\).  Lemma~\ref{lem:trace-bookkeeping} gives a uniform complexity bound for these traces.

Each index belongs to \(O(L)\) rectangles on its unique heavy path.  Each \(a\in A\) meets at most \(O(L)\) heavy paths and belongs to \(O(L)\) rectangles on each such path.  Therefore
\[
 \sum_\nu|I_\nu|\leq C|I|L,
 \qquad
 \sum_\nu|A_\nu|\leq C|A|L^2.
\]
Since \(A_\nu\subseteq A\), monotonicity of the coordinate-hull measure gives
\[
 \sum_\nu |I_\nu|\delta_{r-2}^{\,r-1}(A_\nu)
 \leq
 C|I|\delta_{r-2}^{\,r-1}(A)L.
\]
This proves \eqref{eq:one-sided-packing} for \(D_I\).

For \(D_A\), define
\[
 C_a:=\{z\in V:v(F(a)+z)\triangleright v(G(a))\}.
\]
Discard the empty balls.  Then \((a,i)\in D_A\) if and only if \(M(i)\in C_a\).  Form the weighted inclusion forest of the distinct balls \(C_a\), now assigning to a node the number of parameters \(a\) in its subtree.  On a heavy path
\(C_1\supsetneq\cdots\supsetneq C_m\), let \(\operatorname{idx}_{\mathcal Q}(a)\) be the position of \(C_a\), and define
\[
 \operatorname{dep}_{\mathcal Q}(i)
 :=
 \max\{j:M(i)\in C_j\},
\]
with value \(0\) when no such \(j\) exists.  The relation is again
\(\operatorname{idx}_{\mathcal Q}(a)\leq\operatorname{dep}_{\mathcal Q}(i)\).

The first-coordinate side of every dyadic rectangle is the preimage of an interval \([p,q]\).  By Lemma~\ref{lem:path-interval-recognition}, membership in that interval is equivalent to at most three fixed-point membership or non-membership conditions for the parametrized ball \(C_a\).  Lemma~\ref{lem:trace-bookkeeping} therefore yields a uniform complexity bound.  Each \(a\) occurs in \(O(L)\) rectangles on its unique heavy path, while each \(i\) meets \(O(L)\) heavy paths and occurs in \(O(L)\) rectangles on each.  Thus
\[
 \sum_\nu|A_\nu|\leq C|A|L,
 \qquad
 \sum_\nu|I_\nu|\leq C|I|L^2.
\]
Monotonicity of \(\delta_{r-2}^{\,r-1}\) proves \eqref{eq:one-sided-packing} for \(D_A\).
\end{proof}


Let \(S=P(A)\cup Q(I)\), and let \(\mathcal T(S)\) be its finite ultrametric ball tree.  For a node \(b\), set
\[
 A_b:=\{a\in A:P(a)\in b\},
 \qquad
 I_b:=\{i\in I:Q(i)\in b\},
 \qquad
 w(b):=\Delta(A_b,I_b).
\]
Let \(\operatorname{Ch}(b)\) denote the set of children of \(b\). Then
\begin{equation}\label{eq:child-subadditivity}
 \sum_{c\in\operatorname{Ch}(b)}w(c)\leq w(b).
\end{equation}
Indeed, the sets \(A_c\) and \(I_c\) partition \(A_b\) and \(I_b\), respectively, and
\(\delta_{r-2}^{\,r-1}(A_c)\leq\delta_{r-2}^{\,r-1}(A_b)\).

\begin{lem}[packing of heavy-path tops]
\label{lem:heavy-path-top-packing}
At each non-leaf node choose a child of maximal \(w\)-weight as the heavy child, and let \(\mathfrak P\) be the resulting family of maximal heavy paths.  Then
\begin{equation}\label{eq:path-top-packing}
 \sum_{\mathcal Q\in\mathfrak P}
 w(\operatorname{top}\mathcal Q)
 \leq
 C\,w(\operatorname{root})L(A,I).
\end{equation}
\end{lem}

\begin{proof}
By \eqref{eq:child-subadditivity}, every light child \(c\) of a node \(b\) satisfies \(w(c)\leq w(b)/2\).  Hence a root-to-leaf path contains at most \(1+\lfloor\log_2 w(\operatorname{root})\rfloor\) light edges.  Group the heavy-path tops by the number of light edges on the path from the root.  Tops in one group are pairwise incomparable.  Repeated application of \eqref{eq:child-subadditivity} shows that their total weight is at most the root weight.  Summation over the light-depth groups proves \eqref{eq:path-top-packing}.  Paths of weight zero may be omitted.
\end{proof}

\begin{lem}[Unique heavy-path classification of a pair]
\label{lem:pair-classification}
Let
\[
 \mathcal Q:b_0\supsetneq b_1\supsetneq\cdots\supsetneq b_m
\]
be a maximal heavy path.  For \(0\leq j<m\), put
\[
 A_j:=A_{b_j}\setminus A_{b_{j+1}},
 \qquad
 I_j:=I_{b_j}\setminus I_{b_{j+1}},
\]
and put \(A_m:=A_{b_m}\), \(I_m:=I_{b_m}\).  Every pair \((a,i)\in A\times I\) is assigned to the unique heavy path containing the least common ancestor of \(P(a)\) and \(Q(i)\).  For that path, exactly one of the following alternatives holds:
\begin{enumerate}
 \item \(a\in A_p\), \(i\in I_q\), and \(p<q\);
 \item \(a\in A_p\), \(i\in I_q\), and \(q<p\);
 \item \(a\in A_j\), \(i\in I_j\), and \(P(a),Q(i)\) lie in distinct light children of \(b_j\);
 \item \(a\in A_m\), \(i\in I_m\), and \(P(a)=Q(i)\).
\end{enumerate}
Pairs entering the same light child at the same layer are assigned to the heavy path beginning at that light child.
\end{lem}

\begin{proof}
Follow \(P(a)\) and \(Q(i)\) down the heavy path containing their least common ancestor.  Let \(p\) and \(q\) be the first layers at which the two points fail to enter the heavy child, with value \(m\) if no such layer exists.  Unequal exit layers give (1) or (2).  If \(p=q<m\), the least common ancestor is \(b_p\) exactly when the two points enter distinct light children; if they enter the same light child, their least common ancestor lies in that light subtree.  If \(p=q=m\), both points lie in the singleton leaf \(b_m\).
\end{proof}


A \emph{fixed-radius literal} is a condition
\[
 v(F(a)+M(i))\ \square\ v(\eta),
 \qquad
 \square\in\{<,\leq,>,\geq\},
\]
where \(\eta\in V\) is fixed.  A \emph{moving literal} is a condition
\[
 v(f(a,i))\triangleright v(g(a,i)),
 \qquad
 \triangleright\in\{\geq,>\},
\]
for which \(g\) is nonconstant.

\begin{prop}[Split elimination of one moving literal]
\label{prop:moving-literal-elimination}
Let \(A\) be a semilinear trace of complexity at most \(\kappa\), let \(I\subseteq B_r\), and let
\[
 D:=\{(a,i):v(f(a,i))\triangleright v(g(a,i))\},
 \qquad
 \triangleright\in\{\geq,>\}.
\]
There is an integer \(N\geq0\) and a disjoint decomposition
\begin{equation}\label{eq:moving-elimination-union}
 D\cap(A\times I)
 =
 \bigsqcup_{\nu=1}^{N}
 \left(
  (A_\nu\times I_\nu)
  \cap
  \bigcap_{\ell=1}^{m_\nu}S_{\nu,\ell}
 \right)
\end{equation}
such that:
\begin{enumerate}
 \item \(m_\nu\leq2\), and every \(S_{\nu,\ell}\) is a fixed-radius literal;
 \item every \(A_\nu\) has description complexity bounded in terms of \(\kappa\);
 \item
 \begin{equation}\label{eq:moving-elimination-packing}
  \sum_{\nu=1}^{N}\Delta(A_\nu,I_\nu)
  \leq
  C_\kappa\,\Delta(A,I)L(A,I)^4.
 \end{equation}
\end{enumerate}
\end{prop}

\begin{proof}
If \(A=\varnothing\) or \(I=\varnothing\), take the empty decomposition.  Write
\[
 f(a,i)=F(a)+M(i),
 \qquad
 g(a,i)=P(a)-Q(i),
\]
after absorbing the affine constants into the one-sided terms.  If both \(P\) and \(Q\) are constant, then \(D\) is a single fixed-radius literal.  If exactly one of them is constant, Lemma~\ref{lem:one-sided-rectangularization} gives \eqref{eq:moving-elimination-union} with no remaining fixed literal.  We therefore assume that both \(P\) and \(Q\) are nonconstant.

Let \(S=P(A)\cup Q(I)\), form \(\mathcal T(S)\), and choose heavy children using the weight \(w(b)=\Delta(A_b,I_b)\).  Fix a maximal heavy path
\[
 \mathcal Q:b_0\supsetneq b_1\supsetneq\cdots\supsetneq b_m.
\]
The leaf \(b_m\) is a singleton; denote its point by \(\zeta_{\mathcal Q}\).  Use the sets \(A_j,I_j\) from Lemma~\ref{lem:pair-classification}.  We decompose the pairs assigned to \(\mathcal Q\) according to that lemma.

Suppose first that \(a\in A_p\), \(i\in I_q\), and \(p<q\).  The points \(Q(i)\) and \(\zeta_{\mathcal Q}\) lie in the heavy child \(b_{p+1}\), while \(P(a)\) lies in another child of \(b_p\).  Hence
\begin{equation}\label{eq:exit-identity-A}
 v(P(a)-Q(i))
 =
 v(P(a)-\zeta_{\mathcal Q}).
\end{equation}
Apply Lemma~\ref{lem:dyadic-halfgraph-section6} to the strict semigraph \(p<q\).  The preimage of an interval \([u,v]\) on the \(p\)-side is
\[
 \bigcup_{p=u}^{v}A_p
 =
 A_{b_u}\setminus A_{b_{v+1}},
\]
with the inner ball omitted when \(v=m\); hence it is a bounded-complexity trace.  The first dyadic decomposition has total split measure at most
\(C\Delta(A_{b_0},I_{b_0})L(A,I)\).  On each resulting rectangle, \eqref{eq:exit-identity-A} turns \(D\) into a one-sided relation whose radius depends only on \(a\).  Lemma~\ref{lem:one-sided-rectangularization} then gives disjoint rectangles of total split measure at most
\begin{equation}\label{eq:path-less-charge}
 C\Delta(A_{b_0},I_{b_0})L(A,I)^3.
\end{equation}
The moving literal is automatic on those rectangles.  The case \(q<p\) is symmetric and satisfies the same estimate.

Suppose next that \(a\in A_j\), \(i\in I_j\), and the two images enter distinct light children of \(b_j\).  Let \(\gamma_j\) be the radius level of \(b_j\).  Choose \(u_j\) in the heavy child and \(v_j\) in a light child, and put \(r_j=u_j-v_j\); then \(v(r_j)=\gamma_j\).  On \(A_j\times I_j\), one always has \(v(P(a)-Q(i))\geq\gamma_j\), and the two points enter distinct children exactly when
\(v(P(a)-Q(i))\leq v(r_j)\).  Thus the contribution at level \(j\) is
\begin{equation}\label{eq:equal-exit-piece}
 (A_j\times I_j)
 \cap
 \{v(P(a)-Q(i))\leq v(r_j)\}
 \cap
 \{v(f(a,i))\triangleright v(r_j)\}.
\end{equation}
This piece contains two fixed-radius literals.  Moreover,
\[
 \sum_{j=0}^{m-1}\Delta(A_j,I_j)
 \leq
 \Delta(A_{b_0},I_{b_0}).
\]

Finally, on \(A_m\times I_m\), one has \(P(a)=Q(i)=\zeta_{\mathcal Q}\).  If \(\triangleright\) is \(>\), this part is empty.  If \(\triangleright\) is \(\geq\), it is defined by the fixed literal \(f(a,i)=0\), equivalently \(v(f(a,i))\geq v(0)\).

Lemma~\ref{lem:pair-classification} shows that these pieces are disjoint and exhaust the pairs assigned to \(\mathcal Q\).  By \eqref{eq:path-less-charge} and its symmetric counterpart, one heavy path contributes total split measure at most
\[
 C\Delta(A_{b_0},I_{b_0})L(A,I)^3.
\]
Lemma~\ref{lem:heavy-path-top-packing} supplies one further logarithmic factor, proving \eqref{eq:moving-elimination-packing}.  Every first-side trace introduced above is obtained by a bounded number of fixed-ball, ball-difference, or parametrized-ball membership conditions; Lemma~\ref{lem:trace-bookkeeping} gives the asserted complexity bound.
\end{proof}


\begin{lem}[Quotient-label form of a fixed-radius literal]
\label{lem:quotient-label-form}
Let \(\gamma\in\Gamma\) be finite, and set
\[
 H_{\geq\gamma}:=\{z\in V:v(z)\geq\gamma\},
 \qquad
 H_{>\gamma}:=\{z\in V:v(z)>\gamma\}.
\]
For affine maps \(F\) on \(A\) and \(M\) on \(I\), the following equivalences hold:
\begin{align*}
 v(F(a)+M(i))\geq\gamma
 &\Longleftrightarrow
 F(a)+H_{\geq\gamma}=-M(i)+H_{\geq\gamma},\\
 v(F(a)+M(i))>\gamma
 &\Longleftrightarrow
 F(a)+H_{>\gamma}=-M(i)+H_{>\gamma},\\
 v(F(a)+M(i))<\gamma
 &\Longleftrightarrow
 F(a)+H_{\geq\gamma}\neq-M(i)+H_{\geq\gamma},\\
 v(F(a)+M(i))\leq\gamma
 &\Longleftrightarrow
 F(a)+H_{>\gamma}\neq-M(i)+H_{>\gamma}.
\end{align*}
\end{lem}

\begin{proof}
The first equivalence is the definition of equality modulo \(H_{\geq\gamma}\), and the second is the corresponding statement modulo \(H_{>\gamma}\).  Taking complements gives the final two equivalences.
\end{proof}

For the threshold \(\infty\), the four comparisons are treated separately: \(v(h)\geq\infty\) is \(h=0\), \(v(h)>\infty\) is contradictory, \(v(h)<\infty\) is \(h\neq0\), and \(v(h)\leq\infty\) is tautological.

\begin{prop}[Fixed-radius literal base]
\label{prop:fixed-literal-base}
Assume that the split-grid bound has been established in arity \(r-1\) for every fixed description complexity.  Let \(A\subseteq B_1\times\cdots\times B_{r-1}\) be a semilinear trace of complexity at most \(\kappa\), and let \(I\subseteq B_r\).  Suppose that \(E\subseteq A\times I\) is defined by a conjunction of at most \(m\) fixed-radius literals.  If \(E\) is \(K_{t,\ldots,t}\)-free, then
\begin{equation}\label{eq:fixed-base-bound}
 |E|
 \leq
 C_{r,t,\kappa,m}\,\Delta(A,I)L(A,I)^{d_{r,\kappa,m}}.
\end{equation}
\end{prop}

\begin{proof}
Remove the tautological and contradictory literals.  At threshold \(\infty\), the conditions \(h=0\) and \(h\neq0\) are respectively a label equality and a label inequality modulo the subgroup \(\{0\}\).  By Lemma~\ref{lem:quotient-label-form}, every remaining literal is therefore either a label equality or a label inequality of the form \(\lambda(a)\neq\mu(i)\) in a quotient group.

We first prove a uniform statement for a conjunction of \(d\) label inequalities
\begin{equation}\label{eq:pure-label-inequality}
 E_d
 :=
 \{(a,i)\in A\times I:
   \lambda_k(a)\neq\mu_k(i)
   \text{ for }1\leq k\leq d\},
\end{equation}
where the label maps arise from fixed-radius affine comparisons.  The proof is by strong induction on \(d\), simultaneously over all fixed trace-complexity bounds.

For \(d=0\), the relation is \(A\times I\).  If \(|I|<t\), then
\[
 |A\times I|\leq(t-1)|A|\leq(t-1)\Delta(A,I).
\]
Assume \(|I|\geq t\).  If \(A\) contained a complete \((r-1)\)-partite hypergraph with \(t\) vertices in every part, adjoining any \(t\) points of \(I\) would produce a \(K_{t,\ldots,t}\) in \(A\times I\).  Hence \(A\) is \(K_{t,\ldots,t}\)-free in arity \(r-1\).  The arity induction hypothesis gives
\[
 |A|
 \leq
 C\,\delta_{r-2}^{\,r-1}(A)L(A,I)^c.
\]
The logarithmic parameter in the lower-arity estimate is at most \(L(A,I)\), because \(\delta_{r-2}^{\,r-1}(A)\leq\Delta(A,I)\).  Multiplication by \(|I|\) yields the required estimate for \(E_0\).

Let \(d\geq1\), and assume the assertion for fewer label inequalities.  If \(A\) is already \(K_{t,\ldots,t}\)-free in arity \(r-1\), the same lower-arity estimate bounds \(|E_d|\leq|A||I|\).  Otherwise choose
\[
 T=T_1\times\cdots\times T_{r-1}\subseteq A,
 \qquad |T_j|=t.
\]
For \(1\leq k\leq d\), put \(C_k:=\lambda_k(T)\), so \(|C_k|\leq t^{r-1}\), and define
\[
 I_0
 :=
 \{i\in I:\mu_k(i)\notin C_k
       \text{ for every }k\}.
\]
If \(|I_0|\geq t\), then \(T\times I_0'\subseteq E_d\) for every \(t\)-element subset \(I_0'\subseteq I_0\), a contradiction.  Thus \(|I_0|\leq t-1\).

Order the finite set of pairs \((k,\xi)\) with \(\xi\in C_k\).  For each \(i\in I\setminus I_0\), choose the first pair \((k,\xi)\) with \(\mu_k(i)=\xi\), and let \(I_{k,\xi}\) be the resulting disjoint classes.  Define
\[
 A_{k,\xi}:=\{a\in A:\lambda_k(a)\neq\xi\}.
\]
Choose \(a_{k,\xi}\in T\) with \(\lambda_k(a_{k,\xi})=\xi\).  Then \(A_{k,\xi}\) is defined by a fixed quotient-label inequality with parameter \(a_{k,\xi}\), so Lemma~\ref{lem:trace-bookkeeping} gives a uniform trace-complexity bound.  On \(A_{k,\xi}\times I_{k,\xi}\), the \(k\)-th label inequality is automatic; the remaining relation has \(d-1\) label inequalities and is a subrelation of \(E_d\).  The induction hypothesis therefore gives
\[
 |E_d\cap(A_{k,\xi}\times I_{k,\xi})|
 \leq
 C\Delta(A_{k,\xi},I_{k,\xi})L(A,I)^c.
\]
Since there are at most \(d t^{r-1}\) classes and the sets \(I_{k,\xi}\) are disjoint,
\[
 \sum_{k,\xi}\Delta(A_{k,\xi},I_{k,\xi})
 \leq
 d t^{r-1}|A|
 +
 |I|\delta_{r-2}^{\,r-1}(A)
 \leq
 C_{d,r,t}\Delta(A,I).
\]
The indices in \(I_0\) contribute at most \((t-1)|A|\) edges.  This completes the induction on \(d\).

Return to the original conjunction.  Combine all equality-type literals into one product label
\[
 \lambda^+:A\to\Omega,
 \qquad
 \mu^+:I\to\Omega.
\]
For every label \(\xi\) occurring on both sides, put
\[
 A_\xi:=\{a:\lambda^+(a)=\xi\},
 \qquad
 I_\xi:=\{i:\mu^+(i)=\xi\}.
\]
The matching blocks are disjoint and satisfy
\[
 \sum_\xi\Delta(A_\xi,I_\xi)\leq\Delta(A,I).
\]
For each nonempty \(I_\xi\), choose \(i_\xi\in I_\xi\).  The trace \(A_\xi\) is defined by a bounded conjunction of fixed-radius equalities with the parameter \(i_\xi\), hence has uniformly bounded complexity.  On each block the remaining literals are label inequalities.  The corresponding relation is a subrelation of \(E\), hence is \(K_{t,\ldots,t}\)-free, so the preceding induction applies.  Summing over \(\xi\) proves \eqref{eq:fixed-base-bound}.
\end{proof}


\begin{thm}[Split-grid induction theorem]
\label{thm:split-grid-valuative-induction}
Fix \(r,t\geq2\) and integers \(\kappa,q,m\geq0\).  There exist constants
\(C=C(r,t,\kappa,q,m)\) and \(c=c(r,\kappa,q,m)\) with the following property.

Let \(A\subseteq B_1\times\cdots\times B_{r-1}\) be a semilinear trace of description complexity at most \(\kappa\), and let \(I\subseteq B_r\).  Suppose \(E\subseteq A\times I\) is defined by a conjunction of \(q\) moving literals and \(m\) fixed-radius literals.  If \(E\) is \(K_{t,\ldots,t}\)-free, then
\begin{equation}\label{eq:split-grid-induction-bound}
 |E|
 \leq
 C\,\Delta(A,I)L(A,I)^c.
\end{equation}
\end{thm}

\begin{proof}
We use induction on the arity and, within a fixed arity, induction on the number of moving literals.  The arity-one statement is the elementary bound that a \(K_t\)-free one-partite relation has at most \(t-1\) points.

Fix \(r\geq2\), and assume the theorem in every smaller arity and for every fixed description complexity.  The induction on \(q\) is simultaneous over all fixed values of \(\kappa\) and \(m\).

If \(q=0\), Proposition~\ref{prop:fixed-literal-base} applies, using the arity \(r-1\) induction hypothesis.

Assume \(q\geq1\), choose a moving literal \(D\), and write \(E=E'\cap D\).  Proposition~\ref{prop:moving-literal-elimination} gives
\[
 D\cap(A\times I)
 =
 \bigsqcup_\nu
 \left(
  (A_\nu\times I_\nu)
  \cap
  \bigcap_{\ell=1}^{m_\nu}S_{\nu,\ell}
 \right),
 \qquad m_\nu\leq2.
\]
Consequently, \(E\) is the disjoint union of relations \(E_\nu\), each defined on \(A_\nu\times I_\nu\) by \(q-1\) moving literals and at most \(m+2\) fixed-radius literals.  Every \(E_\nu\) is a subrelation of \(E\), hence is \(K_{t,\ldots,t}\)-free.  The inner induction hypothesis, quantified over all fixed trace complexities and numbers of fixed literals, yields
\[
 |E_\nu|
 \leq
 C'\Delta(A_\nu,I_\nu)L(A,I)^{c'}.
\]
The complexity of every \(A_\nu\) is bounded in terms of \(\kappa\), and \(L(A_\nu,I_\nu)\leq L(A,I)\).  Using \eqref{eq:moving-elimination-packing},
\[
 |E|
 \leq
 C'L(A,I)^{c'}
 \sum_\nu\Delta(A_\nu,I_\nu)
 \leq
 C''\Delta(A,I)L(A,I)^{c'+4}.
\]
This proves \eqref{eq:split-grid-induction-bound}.  The induction is well-founded because a moving-literal step decreases \(q\), while the fixed-literal base invokes only the theorem in arity \(r-1\).  All complexity increases are bounded in terms of the fixed parameters.  The logarithmic exponent is independent of \(t\): the fixed-literal base inherits the exponent from the lower arity, and each moving-literal elimination adds four; the parameter \(t\) affects only multiplicative constants.  The induction on quotient-label inequalities changes only the multiplicative constant; every logarithmic increase comes from Proposition~\ref{prop:moving-literal-elimination} or from the lower-arity exponent.  Consequently, the exponent is independent of \(t\).
\end{proof}


\begin{thm}[Absolute upper bound for valued semilinear hypergraphs]
\label{thm:hypergraph-semilinear}
Fix \(r,t\geq2\) and a description-complexity bound \((\rho,s)\).  There exist constants \(C=C(r,t,\rho,s)\) and \(c=c(r,s)\) such that the following holds.  Let
\[
 B_j\subseteq V^{d_j}
 \qquad(1\leq j\leq r)
\]
be finite, and let
\(E\subseteq B_1\times\cdots\times B_r\) be induced by a semilinear relation of description complexity at most \((\rho,s)\).  If \(E\) is \(K_{t,\ldots,t}\)-free, then
\begin{equation}\label{eq:absolute-delta-bound}
 |E|
 \leq
 C\,
 \delta_{r-1}^{\,r}(B_1\times\cdots\times B_r)
 \left(
  1+
  \log\!
  \left(
   2+\delta_{r-1}^{\,r}(B_1\times\cdots\times B_r)
  \right)
 \right)^c.
\end{equation}
In particular, if \(n=\sum_{j=1}^r|B_j|\), then
\[
 |E|
 =
 O_{r,t,\rho,s}\!\left(n^{r-1}(\log n)^c\right).
\]
\end{thm}

\begin{proof}
Write \(E\) as a union of at most \(\rho\) conjunctive cells, each containing at most \(s\) affine valuation comparisons.  Every cell is a subrelation of \(E\), hence is \(K_{t,\ldots,t}\)-free.  It therefore suffices to treat one cell.

Normalize each comparison by interchanging its two sides when necessary.  If one side is constant, regard the comparison as a fixed-radius literal.  Otherwise regard it as a moving literal.  Thus the cell contains \(q\) moving literals and \(m\) fixed-radius literals with \(q+m\leq s\).

Set
\[
 A:=B_1\times\cdots\times B_{r-1},
 \qquad I:=B_r.
\]
The trace \(A\) has complexity zero.  Theorem~\ref{thm:split-grid-valuative-induction} gives
\[
 |E|
 \leq
 C\Delta(A,I)L(A,I)^c.
\]
Identity \eqref{eq:full-grid-identity} gives \eqref{eq:absolute-delta-bound}.  Finally,
\[
 \delta_{r-1}^{\,r}(B_1\times\cdots\times B_r)
 =
 \sum_{j=1}^r\prod_{\ell\neq j}|B_\ell|
 \leq
 r n^{r-1},
\]
which proves the final estimate.  Summation over the at most \(\rho\) cells only changes the constant.
\end{proof}

\begin{cor}[Definable additive-valuative relations]
\label{cor:definable-valuative-hypergraphs}
Assume the quantifier-elimination statement of Corollary~\ref{cor:qe-semilinear}.  Let \(K=\mathbb Q_p\) or \(K=\mathbb C_p\) in the additive affine-valuative language, and fix a definable \(r\)-partite relation \(R\).  Every finite \(K_{t,\ldots,t}\)-free trace
\[
 E=R\cap(B_1\times\cdots\times B_r)
\]
satisfies
\[
 |E|
 =
 O\!\left(n^{r-1}(\log n)^c\right),
\]
where \(n=\sum_j|B_j|\), and the constants depend on the defining formula, \(r\), and \(t\).
\end{cor}

\begin{proof}
By Corollary~\ref{cor:qe-semilinear}, the defining formula is equivalent to a semilinear formula of finite description complexity.  Apply Theorem~\ref{thm:hypergraph-semilinear}.
\end{proof}

\section{Lower bounds over valued fields}
\label{sect:lower-bound}
\label{sect: lower_bound}

We construct $K_{2,2}$-free point--box incidence graphs with
\[
 \Omega\!\left(n\frac{\log n}{\log\log n}\right)
\]
edges. The construction is an ordered-coordinate version of the product construction of Basit--Chernikov--Starchenko--Tao--Tran~\cite[Lemma~3.3 and Proposition~3.5]{basit2021zarankiewicz}.

\begin{defn}\label{def:point-box}
Let $(K,v)$ be a valued field. For $u,w\in(K^\times)^d$ with
\[
 v(u_j)<v(w_j)
 \qquad(1\leq j\leq d),
\]
the associated \emph{open valuative box} is
\[
 \mathcal B(u,w)
 :=
 \{x\in K^d:v(u_j)<v(x_j)<v(w_j)\text{ for every }j\}.
\]
The point--box incidence relation is
\[
 I_d(x;u,w)
 :=
 \bigwedge_{j=1}^{d}
 \bigl(v(u_j)<v(x_j)\wedge v(x_j)<v(w_j)\bigr).
\]
It is semilinear of description complexity $(1,2d)$.
\end{defn}

For a linearly ordered set $\Lambda$, an open axis-parallel box in $\Lambda^d$ is a product
\[
 \prod_{j=1}^{d}(a_j,b_j)_\Lambda.
\]
Incidence with such boxes depends only on the coordinate-wise order of the finitely many point coordinates and box endpoints.

\begin{lem}[Order-preserving transfer]\label{lem:order-transfer}
Let $P\subseteq\Lambda^d$ be finite and let $\mathcal B$ be a finite family of open axis-parallel boxes. For each coordinate $j$, let $S_j$ be the finite set consisting of the $j$-th coordinates of points of $P$ and the $j$-th endpoints of boxes in $\mathcal B$. If
\[
 \theta_j:S_j\longrightarrow\Lambda'
\]
is strictly order preserving for every $j$, then applying $\theta_j$ in the $j$-th coordinate to all points and endpoints produces a point--box configuration in $(\Lambda')^d$ with an isomorphic incidence graph.
\end{lem}

\begin{proof}
For a point $x$ and a box $\prod_j(a_j,b_j)$, incidence is the conjunction of the inequalities $a_j<x_j<b_j$. Each of these inequalities is preserved by $\theta_j$.
\end{proof}

\begin{lem}[General-position realization]\label{lem:point-box-general-position}
Let $P\subseteq\mathbb R^d$ be finite and let $\mathcal B$ be a finite
family of bounded open axis-parallel boxes. There is an
incidence-isomorphic configuration $(\widetilde P,\widetilde{\mathcal B})$
in which, in each coordinate, the point coordinates are pairwise distinct
and no point coordinate is a box endpoint.
\end{lem}

\begin{proof}
Treat one coordinate at a time. Introduce a labelled symbol for every
point coordinate and for every lower and upper endpoint. Symbols with
different original values retain their original order. Among symbols with
the same original value, place upper-endpoint symbols first, point symbols
next in an arbitrary strict order, and lower-endpoint symbols last. Assign
strictly increasing rational numbers to the resulting finite total order.

If $a$ is a lower endpoint and $x$ a point coordinate, then the new
inequality $\widetilde a<\widetilde x$ holds exactly when $a<x$ held
originally: when $a=x$, the point symbol was placed before the lower
endpoint. Similarly, $\widetilde x<\widetilde b$ holds exactly when
$x<b$ for an upper endpoint $b$, because upper endpoints precede points
at a tie. Thus all point--box incidences and non-incidences are preserved.
The lower and upper endpoints of each box remain correctly ordered, since
their original values were distinct. Repeating this construction in every
coordinate proves the lemma.
\end{proof}

\begin{lem}[Product construction]\label{lem:lower-product}
Let $d\geq2$. Suppose there is a $K_{2,2}$-free incidence configuration of $n_1$ points and $n_2$ bounded open boxes in $\mathbb R^{d-1}$ with $m$ incidences, and a $K_{2,2}$-free incidence configuration of $n_1'$ points and $n_2'$ bounded open boxes in $\mathbb R^d$ with $m'$ incidences. Then there is a $K_{2,2}$-free incidence configuration in $\mathbb R^d$ with
\[
 n_1n_1'\text{ points},
 \qquad
 n_1n_2'+n_1'n_2\text{ boxes},
\]
and
\[
 n_1m'+n_1'm
\]
incidences. The statement permits $n_2'=m'=0$.
\end{lem}

\begin{proof}
By Lemma~\ref{lem:point-box-general-position}, applied to both input configurations, we may assume that no point lies on a box boundary and that the point coordinates are pairwise distinct in each coordinate.

Write the lower-dimensional point set as
\[
 P^{d-1}=\{p_1,\ldots,p_{n_1}\}.
\]
Choose pairwise disjoint open boxes $\beta_i\subseteq\mathbb R^{d-1}$ containing $p_i$ so small that, for every choice $p_i'\in\beta_i$, the incidence pattern of $\{p_1',\ldots,p_{n_1}'\}$ with the lower-dimensional boxes is unchanged. We also require every side length of every $\beta_i$ to be smaller than every side length occurring among the lower-dimensional boxes. These choices are possible because only finitely many strict coordinate inequalities and finitely many positive side lengths are involved.

For each $i$, apply positive affine contractions in the first $d-1$ coordinates, leaving the last coordinate fixed, to place a copy
\[
 (P_i^d,\mathcal B_i^d)
\]
of the $d$-dimensional configuration inside $\beta_i\times\mathbb R$. Put
\[
 P:=\bigsqcup_{i=1}^{n_1}P_i^d,
 \qquad
 \mathcal B':=\bigsqcup_{i=1}^{n_1}\mathcal B_i^d.
\]
These sets contribute $n_1n_1'$ points, $n_1n_2'$ boxes, and $n_1m'$ incidences. Boxes belonging to different copies are disjoint in their first $d-1$ coordinates.

Write the points of the original $d$-dimensional configuration as $q_1,\ldots,q_{n_1'}$. Their last coordinates are pairwise distinct. Choose pairwise disjoint open intervals $J_j$ containing the last coordinate of $q_j$ and no last coordinate of any other $q_{j'}$. For every lower-dimensional box $C$ and every $j$, form the $d$-dimensional box
\[
 C\times J_j.
\]
Let $\mathcal C_j$ be the family of these $n_2$ boxes. In each cluster $P_i^d$, exactly one point has last coordinate in $J_j$; denote it by $q_{i,j}$. Its first $d-1$ coordinates lie in $\beta_i$. Consequently the incidence graph between $P$ and $\mathcal C_j$ is isomorphic to the lower-dimensional incidence graph and contributes $m$ incidences. The families $\mathcal C_j$ are pairwise disjoint because the intervals $J_j$ are disjoint. No global box can coincide with a local box: the first-coordinate projection of every local box is contained in some $\beta_i$, whereas every global box has the first-coordinate projection of an original lower-dimensional box, whose side lengths are strictly larger.

Set
\[
 \mathcal B:=\mathcal B'\sqcup\bigsqcup_{j=1}^{n_1'}\mathcal C_j.
\]
The asserted counts follow. It remains to prove $K_{2,2}$-freeness. Two boxes from $\mathcal B'$ either lie in the same copy, where the original $d$-dimensional graph is $K_{2,2}$-free, or in different copies, where they are disjoint. Two boxes from the same $\mathcal C_j$ reduce to the lower-dimensional graph, and boxes from different $\mathcal C_j$ are disjoint. Finally, if one box lies in $\mathcal B_i^d$ and the other in $\mathcal C_j$, then any common point lies in the cluster $P_i^d$ and has last coordinate in $J_j$; there is exactly one such point, namely $q_{i,j}$. Thus no two distinct points lie in two distinct boxes.
\end{proof}

\begin{prop}[Point--box lower bound in an ordered value group]\label{prop:value-group-lower-bound}
For every integer $\ell\geq2$, there is a $K_{2,2}$-free incidence configuration in $\mathbb Z^2$ with
\[
 \ell^\ell\text{ points},
 \qquad
 \ell^\ell\text{ open boxes},
 \qquad
 \ell^{\ell+1}\text{ incidences}.
\]
The same finite incidence graph is realizable in $\Gamma^2$ for every infinite ordered abelian group $\Gamma$.
\end{prop}

\begin{proof}
Take in $\mathbb R$ the $\ell$ points $1,\ldots,\ell$ and the single interval $(0,\ell+1)$. This is a $K_{2,2}$-free configuration with parameters
\[
 n_1=\ell,
 \qquad
 n_2=1,
 \qquad
 m=\ell.
\]
Start in $\mathbb R^2$ with one point and no boxes. Repeatedly apply Lemma~\ref{lem:lower-product}. If $a_j,b_j,e_j$ denote the numbers of points, boxes, and incidences after $j$ iterations, then
\[
 a_{j+1}=\ell a_j,
 \qquad
 b_{j+1}=\ell b_j+a_j,
 \qquad
 e_{j+1}=\ell e_j+\ell a_j,
\]
with $a_0=1$ and $b_0=e_0=0$. Induction gives
\[
 a_j=\ell^j,
 \qquad
 b_j=j\ell^{j-1},
 \qquad
 e_j=j\ell^j.
\]
At $j=\ell$ these are $\ell^\ell,\ell^\ell,\ell^{\ell+1}$.

The construction produces a finite real point--box configuration with no point on a boundary. For each coordinate, replace the finite ordered set of all point coordinates and box endpoints by its rank in $\mathbb Z$. Lemma~\ref{lem:order-transfer} preserves the incidence graph. Finally, any infinite ordered abelian group contains a chain of the required finite length, so a second application of Lemma~\ref{lem:order-transfer} realizes the same graph in $\Gamma^2$.
\end{proof}

\begin{thm}[Valuative point--box lower bound]\label{thm:valuative-lower-bound}
Let $(K,v)$ be a valued field with infinite value group. For arbitrarily large integers $n$, there is a $K_{2,2}$-free semilinear graph with point side in $K^2$, parameter side in $K^4$, total number of vertices $n$, and
\[
 \Omega\!\left(n\frac{\log n}{\log\log n}\right)
\]
edges. The edge relation is $I_2$ from Definition~\ref{def:point-box} and has description complexity $(1,4)$.
\end{thm}

\begin{proof}
Let $\Gamma=v(K^\times)$. Apply Proposition~\ref{prop:value-group-lower-bound}, and choose for every value appearing in the configuration an element of $K^\times$ with that valuation. Map a point $(\gamma_1,\gamma_2)$ to a point $(x_1,x_2)\in K^2$ with $v(x_j)=\gamma_j$, and map a box $\prod_j(\alpha_j,\beta_j)$ to parameters $u,w\in(K^\times)^2$ with
\[
 v(u_j)=\alpha_j,
 \qquad
 v(w_j)=\beta_j.
\]
Definition~\ref{def:point-box} shows that the lifted incidence graph is identical.

For the configuration associated with $\ell$, the two vertex classes both have cardinality $N_\ell:=\ell^\ell$, and the edge count is $\ell N_\ell$. Put $n_\ell:=2N_\ell$. Then
\[
 \log n_\ell=\ell\log\ell+O(1),
 \qquad
 \log\log n_\ell=\log\ell+O(\log\log\ell),
\]
so
\[
 \ell\asymp\frac{\log n_\ell}{\log\log n_\ell}.
\]
Consequently
\[
 \ell N_\ell
 \asymp
 n_\ell\frac{\log n_\ell}{\log\log n_\ell}.
\]
The sequence $n_\ell$ tends to infinity, which proves the theorem.
\end{proof}

\begin{cor}\label{cor: lower_bound_for_theorem}
The conclusion of Theorem~\ref{thm:valuative-lower-bound} holds in particular for $K=\mathbb Q_p$ and $K=\mathbb C_p$.
\end{cor}

\printbibliography

@inproceedings{basit2021zarankiewicz,
  title={Zarankiewicz’s problem for semilinear hypergraphs},
  author={Basit, Abdul and Chernikov, Artem and Starchenko, Sergei and Tao, Terence and Tran, Chieu-Minh},
  booktitle={Forum of Mathematics, Sigma},
  volume={9},
  pages={e59},
  year={2021},
  organization={Cambridge University Press}
}

@article{ChernikovMennen,
  author    = {Artem Chernikov and Alex Mennen},
  title     = {Semi-Equational Theories},
  journal   = {J. Symb. Log.},
  volume    = {90},
  number    = {1},
  pages     = {391--422},
  year      = {2025},
  doi       = {10.1017/jsl.2023.28},
  note      = {Published online by Cambridge University Press on behalf of The Association for Symbolic Logic},
  mrnumber  = {----} % (fill in when MR number is available)
}

@inproceedings{kovari1954problem,
  title={On a problem of K. Zarankiewicz},
  author={Kov{\'a}ri, Tam{\'a}s and S{\'o}s, Vera and Tur{\'a}n, P{\'a}l},
  booktitle={Colloquium Mathematicum},
  volume={3},
  number={1},
  pages={50--57},
  year={1954},
  organization={Polska Akademia Nauk. Instytut Matematyczny PAN}
}

@article{erdos1964extremal,
  title={On extremal problems of graphs and generalized graphs},
  author={Erd{\"o}s, Paul},
  journal={Israel Journal of Mathematics},
  volume={2},
  number={3},
  pages={183--190},
  year={1964},
  publisher={Springer}
}

@article{fox2017semi,
  title={A semi-algebraic version of Zarankiewicz's problem.},
  author={Fox, Jacob and Pach, J{\'a}nos and Sheffer, Adam and Suk, Andrew and Zahl, Joshua},
  journal={Journal of the European Mathematical Society (EMS Publishing)},
  volume={19},
  number={6},
  year={2017}
}

@article{basu2018minimal,
  title={An o-minimal Szemer{\'e}di--Trotter theorem},
  author={Basu, Saugata and Raz, Orit E},
  journal={The Quarterly Journal of Mathematics},
  volume={69},
  number={1},
  pages={223--239},
  year={2018},
  publisher={Oxford University Press}
}

@inproceedings{zarankiewicz1951problem,
  title={Problem p 101},
  author={Zarankiewicz, Kazimierz},
  booktitle={Colloq. Math},
  volume={2},
  number={301},
  pages={5},
  year={1951}
}

@article{chernikov2020cutting,
  title={Cutting lemma and Zarankiewicz’s problem in distal structures},
  author={Chernikov, Artem and Galvin, David and Starchenko, Sergei},
  journal={Selecta Mathematica},
  volume={26},
  number={2},
  pages={25},
  year={2020},
  publisher={Springer}
}

@article{szemeredi1983extremal,
  author  = {Szemer{\'e}di, Endre and Trotter, William Thomas},
  title   = {Extremal problems in discrete geometry},
  journal = {Combinatorica},
  volume  = {3},
  pages   = {381--392},
  year    = {1983},
}

\end{document}